\documentclass[reqno]{amsart}
\usepackage{graphicx,graphics}
\usepackage{subfig}

\newcommand{\D}{\mathcal{D}}

\newcommand{\Real}{\mathbb R}
\newcommand{\abs}[1]{\left\vert#1\right\vert}

\newtheorem{lem}{Lemma}
\newtheorem{thm}{Theorem}

\newtheorem{rem}{Remark}
\newtheorem{defn}{Definition}

\numberwithin{thm}{section}
\numberwithin{lem}{section}
\numberwithin{coll}{section}
\numberwithin{rem}{section}
\numberwithin{exm}{section}
\numberwithin{prop}{section}
\numberwithin{equation}{section}

\numberwithin{equation}{section}

\begin{document}

\centerline {\textsc {\large On convergence of entropy of distribution functions}} 
\centerline {\textsc {\large in the max domain of attraction of max stable laws}} 
\vspace{0.5in}

\begin{center}
   Sreenivasan Ravi\footnote{Corresponding author: ravi@statistics.uni-mysore.ac.in, sreenivasanravi@yahoo.com} and
   Ali Saeb\footnote{ali.saeb@gmail.com}\\
   Department of Studies in Statistics\\
University of Mysore\\
Manasagangotri, Mysore 570006, India
\end{center}

\vspace{1in}



\noindent {\bf Abstract:}	Max stable laws are limit laws of linearly normalized partial maxima of independent, identically distributed (iid) random variables (rvs). These are analogous to stable laws which are limit laws of normalized partial sums of iid rvs. In this paper, we study entropy limit theorems for distribution functions in the max domain of attraction of max stable laws under linear normalization. More specifically, we study the problem of convergence of the Shannon entropy of linearly normalized partial maxima of iid rvs to the corresponding limit entropy when the linearly normalized partial maxima converges to some nondegenerate rv. We are able to show that the Shannon entropy not only converges but, in fact, increases to the limit entropy in some cases. We discuss several examples. We also study analogous results for the $k$-th upper extremes.

\vspace{0.5in}

\vspace{0.2in} \noindent {\bf Keywords:} Shannon Entropy, Entropy convergence, Relative entropy, Max stable laws, Max domains of attraction, Extreme order statistics.

\vspace{0.5in}

\vspace{0.2in} \noindent {\bf MSC 2010 classification:} 60F10
\newpage
\section{Introduction, Definitions and Preliminary Lemmata}
The limit laws of linearly normalized partial maxima $M_n=X_1\vee\ldots\vee X_n$ of independent and identically distributed (iid) random variables (rvs) $X_1,X_2,\ldots,$ with common distribution function (df) $F,$ namely,
\begin{equation}\label{Introduction_e1}
	\lim_{n\to\infty}P(M_n\leq a_nx+b_n)=\lim_{n\to\infty}F^n(a_nx+b_n)=G(x),\;\;x\in \mathcal{C}(G),
\end{equation}
	where, $a_n>0,$ $b_n\in\Real,$ are norming constants, $G$ is a non-degenerate df, $\mathcal{C}(G)$ is the set of all continuity points of $G,$ are called max stable laws. If, for some non-degenerate df $G,$ a df $F$ satisfies (\ref{Introduction_e1}) for some norming constants $a_n>0,$ $b_n\in\Real,$ then we say that $F$ belongs to the max domain of attraction of $G$ under linear normalization and denote it by $F\in \mathcal{D}(G).$ Limit dfs $G$ satisfying (\ref{Introduction_e1}) are the well known extreme value types of distributions, or max stable laws, namely,
	\begin{eqnarray*}
			\text{the Fr\'{e}chet law:} & \Phi_\alpha(x)  = \left\lbrace	
							\begin{array}{l l}
							 0, &\;\;\; x< 0, \\
							 \exp(-x^{-\alpha}), &\;\;\; 0\leq x;\\
							 \end{array}
							 \right. \\					
		\text{the Weibull law:} & \Psi_\alpha(x) =  \left\lbrace
						\begin{array}{l l} \exp(- |x|^{\alpha}), & x<0, \\
						1, & 0\leq x;
						\end{array}\right. \\
		\text{and the Gumbel law:} & \Lambda(x) = \exp(-\exp(-x));\;\;\;\;\; x\in\Real;
	\end{eqnarray*}
$\alpha>0$ being a parameter, with respective probability density functions (pdfs),
	\begin{eqnarray*}
		\text{the Fr\'{e}chet density:} & \phi_\alpha(x) =  \left\lbrace	
							\begin{array}{l l}
							 0, &\;\;\; x \leq 0, \\
							 \alpha x^{-(\alpha+1)}e^{-x^{-\alpha}}, &\;\;\; 0 < x;\\
							 \end{array}
							 \right. \\
		\text{the Weibull density:} & \psi_\alpha(x) = \left\lbrace
						\begin{array}{l l} \alpha |x|^{\alpha-1}e^{-|x|^{\alpha}}, & x<0, \\
						0, & 0 \leq x;
						\end{array}\right. \\
		\text{and the Gumbel density:} & \lambda(x) = e^{-x}e^{-e^{-x}},\;\; x\in\Real.
	\end{eqnarray*}
Note that (\ref{Introduction_e1}) is equivalent to
\begin{eqnarray}\label{Introduction_e1.2}
\lim_{n\to\infty}n\{1-F(a_nx+b_n)\}=-\log G(x), \; x \in \{y: G(y) > 0\}.
\end{eqnarray}
Criteria for $F\in \mathcal{D}(G)$ are well known (see, for example, Galambos, 1987; Resnick, 1987; Embrechts et al., 1997).

Let $X_{1:n} \leq X_{2:n} \leq \ldots \leq X_{n:n}$ denote the order statistics from a random sample $\;\{X_1, \ldots, X_n\} \;$ from $\;F.\;$ Let the df of the $k$-th extreme be $\;G_{k:n}(x) = P\left( X_{n-k+1:n} \leq a_n x + b_n \right),$ $k = 1, 2, \ldots, \;$ fixed, and let $K_k(x) = \lim_{n\to\infty} G_{k:n}(x).\; $ Then it is well known that if $F \in \mathcal{D}(G)$ for some non-degenerate df $G$ so that (\ref{Introduction_e1}) holds for some norming constants $a_n, b_n,$ then 
\begin{eqnarray}\label{Introduction_e1.3}
K_k(x) = \left\{ \begin{array}{cl} G(x) \sum_{i=0}^{k-1} \frac{(- \log G(x))^i}{i!}, & x \in \{y: G(y) > 0\}, \\ 0, & \text{otherwise}. 
\end{array} \right.
\end{eqnarray}
 We refer to Galambos (1987) and Resnick (1987) for results used in this article. For quick reference, some of the results used in this article are given in Appendix \ref{more results}.

	The study of entropy and relative entropy (as defined later) is important in information theory. Gnedenko and Korolev (1996) suggest that a df that maximizes entropy within a class of dfs often turn out to have favourable properties; for example, the normal df has maximum entropy in the class of dfs having a specified variance. Barron (1986) discusses the central limit theorem in the sense of relative entropy. Johnson (2006) is a good reference to the application of information theory to limit theorems, especially the central limit theorem. In this article, our main interest is to investigate conditions under which the entropy or Shannon's entropy of the normalized partial maxima of iid rvs converges to the corresponding limit entropy. We first look at this problem through some illustrative examples, then at the general case and finally for the $k$-th extremes. Since entropies involve integrals, the proofs of our results here involve clever application of existing results on moment and density convergence  of normalized partial maxima and the dominated convergence theorem. We now give some definitions and preliminary results which will be used subsequently. 
	
\subsection{Some definitions and preliminary results} We will refer to Shannon's entropy as entropy in this article.

	Suppose that (\ref{Introduction_e1}) holds for some df $F$ and some max stable law $G.$ Let $f$ and $g$ respectively denote the pdfs of $F$ and $G.$ The rv $\left( \dfrac{M_n-b_n}{a_n} \right) $ has pdf given by
\begin{equation}\label{Introduction_e2}
	g_n(x)= na_nF^{n-1}(a_nx+b_n)f(a_nx+b_n),\;\;x\in\Real,\;\;n\geq 1.
\end{equation}
\begin{defn}	\label{Definition 1} 
	The entropy of $g_n$ is given by 
\begin{eqnarray*}
	H(g_n) =   -\int_A g_n(x)\log g_n(x)dx, \;\;\mbox{where}\;\;A=\{x\in\Real:g_n(x)>0\}.
	\end{eqnarray*}
\end{defn}

\begin{rem} We have 	
\begin{eqnarray}
	H(g_n)& =  & -\int_A g_n(x)\log g_n(x)dx, \;\;\mbox{where}\;\;A=\{x\in\Real:g_n(x)>0\},\label{Introduction_e3}\\
 & = & - \int_A n(n-1) a_nf(a_nx+ b_n) F^{n-1}(a_nx+b_n) \log F(a_nx+b_n) dx\nonumber\\
 && - \int_A n a_n f(a_nx+ b_n) F^{n-1}(a_nx+b_n) \log(n a_n f(a_nx+b_n)) dx,\nonumber \\
 & = & - (I_1(n) + I_2(n)), \; \mbox{say. We then have} \nonumber\\
 \lim_{n \rightarrow \infty} H(g_n) & = & 1 - \lim_{n\rightarrow \infty} I_2(n), \label{1.5}
\end{eqnarray}
in view of the following lemma. \end{rem} 

\begin{lem}\label{Lemma1}
\begin{eqnarray}
I_1(n) = - \frac{(n-1)}{n}, \, n \geq 1;\;\; \text{and } I_1(n)\text{ decreases to -1 as }n\to\infty.\nonumber
\end{eqnarray}
\end{lem}
\begin{proof} Making the change of variable $F(a_n x + b_n) = t,\;$ we get $\;a_n f(a_n x + b_n) dx = dt\;$ and
$$ I_1(n)  =  \int_{0}^{1} n (n-1) t^{n-1} \log t \,dt  =   - \frac{(n-1)}{n};$$
and the claim follows, since $I_1(n-1)\geq I_1(n).$
\end{proof}

\vspace{0.1in}
Now we state and prove a lemma of independent interest which will be used subsequently.
\begin{lem} \label{Lemma2}
If $Y_1,Y_2,\ldots$ are iid rvs having standard exponential distribution and $Z_n= Y_1 \vee \ldots \vee Y_n,$ then
\begin{eqnarray}
&& \lim_{n\to\infty}E\left(Z_n-\log n\right)=-\Gamma'(1)=\gamma,\nonumber
\end{eqnarray}
	where $\gamma$ is the Euler's constant.
\end{lem}

\begin{proof} As is well known, we have
\[\lim_{n\to\infty}P(Z_n-\log n\leq v)=\lim_{n\to\infty}\left(1-e^{-(v+\log n)}\right)^n=\Lambda(v),\;\;v\in\Real.\]
From the moment convergence result of Proposition (2.1)-(iii) in Resnick (1987) (Theorem B.1), since $\int_{-\infty}^{0}|x|d(1-e^{-x})<\infty,$ we get $ \lim_{n\to\infty}E\left(Z_n-\log n\right)=-\Gamma'(1)=\gamma,$ the Euler's constant.

Alternatively, $Z_n$ has the same distribution as $\sum_{k=1}^{n}\frac{Y_k}{k}$ so that $E(Z_n) = $ $\sum_{k=1}^{n}\frac{1}{k}$ $= \log n + \gamma,$ and the result follows. \end{proof}

\vspace{0.2in}
\begin{defn} The relative entropy of $g_n$ with respect to pdf $g$ is given by
\begin{eqnarray*}
D(g_n\|g) = \int_A g_n(x)\log\left(\dfrac{g_n(x)}{g(x)}\right)dx,\;\;A=\{x\in\Real: g_n(x)>0\} \end{eqnarray*} \end{defn}
\begin{rem} We have 
\begin{eqnarray}\label{Introduction_e4}
	0\leq D(g_n\|g)&=&\int_A g_n(x)\log\left(\dfrac{g_n(x)}{g(x)}\right)dx,\;\;A=\{x\in\Real: g_n(x)>0\},\nonumber\\
	&=&\int_A g_n(x)\log g_n(x)dx-\int_A g_n(x)\log g(x)dx,\nonumber\\
	&=&-H(g_n)-\int_A na_nf(a_nx+b_n)F^{n-1}(a_nx+b_n)\log g(x)dx,\nonumber\\
	&=&-H(g_n)+\Delta_g(g_n),\;\;\text{say,} \end{eqnarray} and we have \begin{eqnarray*}
	\lim_{n \rightarrow \infty}D(g_n\|g) & = & -H(g)+\lim_{n\to\infty}\Delta_{g}(g_n)\text{ if}\; \lim_{n\to\infty} H(g_n) = H(g), \\ 
& = & 0, \;\;\text{if in addition,} \;\; \lim_{n\to\infty}\Delta_g(g_n)=H(g). \nonumber
\end{eqnarray*}
\end{rem}

The entropies of the extreme value distributions are given in the next lemma without proof as the calculations of these are straight forward.
\begin{lem}
The entropy of
\begin{itemize}
\item[(i)] Fr\'{e}chet law: $H(\phi_{\alpha})=-\log\alpha+\frac{\alpha+1}{\alpha}\gamma+1;$
\item[(ii)] Weibull law: $ H(\psi_{\alpha})=-\log\alpha+\frac{\alpha-1}{\alpha}\gamma+1;$
\item[(iii)] Gumbel law: $H(\lambda)=\gamma+1.$
\end{itemize}
	\end{lem}
	We shall denote the left extremity of df $F$ by $l(F) = \inf\{x: F(x) > 0\} \geq - \infty $ and the right extremity of $F$ by $r(F) = \sup \{x: F(x) < 1\} \leq \infty.$ In the next section we give our main results, followed by a section on Proofs. Wherever the proof is similar, we omit the steps and refer to an earlier proof. Finally, we give two appendices containing illustrative graphs and results used in this article.
	
\section{Main Results}
Our first results consider the problem through some illustrative and interesting examples. Though these follow from the general results given later, the proofs of these results are quite different from those of the general results which are stated in the second theorem below. The third theorem below gives results for the $k$-th extremes. In the case of the $k$-th extremes, we do not discuss the monotonicity of the convergence in this article. 
\begin{thm}\label{Thmexm_ent} If $F$ is the
\begin{enumerate}
\item[ (a)] Pareto$(\alpha)$ df so that $F(x) = \left\{ \begin{array}{cl} 0, & \; x \leq 0, \\ 1 - \frac{1}{x^{\alpha}}, & \; 0 < 1; \alpha > 0,\end{array}\right.$ and $\lim_{n \rightarrow \infty} F^n(n^{\frac{1}{\alpha}}x)$ $= \Phi_{\alpha}(x), $ $x \in R;$ then $H(g_n) \uparrow H(\phi_\alpha)$ with $n\;$ and $\;\lim_{n\to\infty}D(g_n\|\phi_\alpha)=0.$
\item[ (b)] uniform df over $(0, 1),$ so that $\lim_{n \rightarrow \infty} F^n(\frac{1}{n}x + 1)$ $= \Psi_{1}(x), $ $x \in R;$ then $H(g_n) \uparrow H(\psi_1)\;$ with $n\;$ and $\;\lim_{n\to\infty}D(g_n\|\psi_1)=0.$
\item[ (c)] standard exponential df so that $\lim_{n \rightarrow \infty} F^n(x + \log n)$ $= \Lambda(x), $ $x \in R;$ then $H(g_n) \uparrow H(\lambda)\;$ with $n\;$ and $\;\lim_{n\to\infty}D(g_n\|\lambda)=0.$
\item[ (d)] standard normal df so that $\lim_{n \rightarrow \infty} F^n(a_n x + b_n)$ $= \Lambda(x), $ $x \in R;$, with $a_n$ and $b_n$ as given in the proof, then $\lim_{n \rightarrow \infty} H(g_n) = H(\lambda)\;$ and $\lim_{n\to\infty}D(g_n\|\lambda)=0.$
\end{enumerate}
\end{thm}

\begin{rem} \begin{enumerate}
\item The results of the above theorem hold for location and scale versions of the dfs also.
\item In the case of (d) above, we notice that $H(g_n)$ does not increase to $H(\lambda)$ with $n,$ as seen in the graph \ref{normal_graph_f} in Appendix \ref{Graphs}.
\end{enumerate}
\end{rem}

\begin{thm}\label{TPhi} Let $F\in\D(G)$ for some nondegenerate df $G,$ with norming constants $a_n$ and $b_n$ so that (\ref{Introduction_e1}) holds and the df $\;F\;$ be absolutely continuous with nonincreasing pdf $\;f\;$ which is eventually positive, that is, $\;f(x) > 0\;$ for $x$ close to $r(F).$
\begin{enumerate}
\item[(a)]	If $\;l(F) > 0\;$ and $\;G=\Phi_\alpha\;$ for some $\;\alpha > 0,\;$ then $\;\lim_{n \rightarrow \infty} H(g_n) = H(\phi_\alpha)\;$ and $\;\lim_{n\to\infty}D(g_n\|\phi_\alpha)=0.$
\item[(b)]	If $\;G=\Psi_\alpha\;$ for some $\;\alpha > 0,\;$  then $\;\lim_{n \rightarrow \infty} H(g_n) = H(\psi_\alpha)\;$ and $\;\lim_{n\to\infty}D(g_n\|\psi_\alpha)=0.$
\item[(c)] If $\;G=\Lambda,\;$ then $\;\lim_{n \rightarrow \infty} H(g_n) = H(\lambda)\;$ and $\;\lim_{n\to\infty}D(g_n\|\lambda)=0.$
\end{enumerate}
Further, if $g_n(x)$ is nonincreasing for $\;n\;$ large, then the entropies above increase to their limits.
\end{thm}

\begin{rem}
\begin{enumerate} \item Note that if $F(x)=0,$ if $x<1,$ and $=1-x^{-\alpha}$ if $1\leq x$ for $\alpha > 0,$ the Pareto df, then with $a_n=n^{\frac{1}{\alpha}},$ $b_n=0,$ the conditions given in the Theorem \ref{TPhi} for $g_n$ are satisfied. In the case of $U(0,1),$ and $Exp(\lambda)$ dfs also these conditions are satisfied.
\item In the case of normal, the condition $g_n(x)$ nonincreasing in Theorem \ref{TPhi} is not satisfied as seen from the graph \ref{normal_graph_ent} in Appendix \ref{Graphs} and also $H(g_n)$ does not increase to $H(\lambda).$ \end{enumerate}
\end{rem}

We need the following lemmata to prove the result on entropy convergence for $k$-th extremes, the first of which gives local uniform convergence for the $k$-th extreme. 

\begin{lem}\label{Lemma.5} Let $F\in\D(G)$ for some nondegenerate df $G,$ with norming constants $a_n$ and $b_n$ so that (\ref{Introduction_e1}) holds and the df $\;F\;$ be absolutely continuous with nonincreasing pdf $\;f\;$ which is eventually positive, that is, $\;f(x) > 0\;$ for $x$ close to $r(F).$ Then the pdf $g_{k:n}$ of $G_{k:n}$ in (\ref{Introduction_e1.3}) converges to the pdf of $K_k,$ locally uniformly.
\end{lem}

\begin{lem}\label{Lemma3} The value of the integral
\begin{eqnarray}
A(k)=\int_{0}^{\infty} u^{k-1}e^{-u}\log u du=(k-1)!\left(-\gamma+\sum_{i=1}^{k-1}\dfrac{1}{i}\right), \; k\geq 2,
\end{eqnarray} with $\;A(1) = - \gamma.\;$
\end{lem}

\begin{lem}\label{TPhi_orderEnt}
The entropy of $K_k$ in (\ref{Introduction_e1.3}) when $G$ is
\begin{itemize}\label{ent_order}
\item[(i)] Fr\'{e}chet law is
		\[H(\phi_{\alpha}^{(k)})=-\log\dfrac{\alpha}{(k-1)!}-\dfrac{\alpha k+1}{\alpha}\left(-\gamma+\sum_{i=1}^{k-1}\dfrac{1}{i}\right)+\dfrac{\Gamma(k+1)}{(k-1)!};\]
\item[(ii)] Weibull law is
		\[H(\psi_\alpha^{(k)})=-\log\dfrac{\alpha}{(k-1)!}-\dfrac{\alpha k-1}{\alpha}\left(-\gamma+\sum_{i=1}^{k-1}\dfrac{1}{i}\right)+\dfrac{\Gamma(k+1)}{(k-1)!};\]
\item[(iii)] Gumbel law is
		\[H(\lambda^{(k)})=\log(k-1)!-k\left(-\gamma+\sum_{i=1}^{k-1}\dfrac{1}{i}\right)+\dfrac{\Gamma(k+1)}{(k-1)!}.\]
\end{itemize}
	\end{lem}

\begin{thm} \label{TPhi_order}
Let $F\in\D(G)$ for some nondegenerate df $G,$ with norming constants $a_n$ and $b_n$ so that (\ref{Introduction_e1}) holds and the df $\;F\;$ be absolutely continuous with nonincreasing pdf $\;f\;$ which is eventually positive, that is, $\;f(x) > 0\;$ for $x$ close to $r(F).$ In (\ref{Introduction_e1.3}), if
\begin{enumerate}
\item[(i)]	$G=\Phi_\alpha$ for some $\;\alpha > 0\;$ with $\;l(F) > 0,\;$ then $\;\lim_{n \rightarrow \infty} H(g_{n:k}) = H(\phi_\alpha^{(k)});$
\item[(ii)]	$G=\Psi_\alpha$ for some $\;\alpha > 0\;$ with $\;r(F) > 0,\;$ then $\;\lim_{n \rightarrow \infty} H(g_{n:k}) = H(\psi_\alpha^{(k)});$
\item[(iii)]	$G=\Lambda,$ then $\;\lim_{n \rightarrow \infty} H(g_{n:k}) = H(\lambda^{(k)}).$
\end{enumerate}
\end{thm}

\section{Proofs}
\begin{proof}[\textbf{Proof of Theorem \ref{Thmexm_ent} (a)}]
	The pdf of Pareto$(\alpha)$ with df $F$ is $ f(x)=\left\lbrace
	\begin{array}{cl}
        0, \;& x\leq 1, \\
		\dfrac{\alpha}{x^{\alpha+1}}, \;& 1 < x. \\
	\end{array}
\right.$ Hence
$g_n(x)=\left(1-\dfrac{1}{nx^\alpha}\right)^{n-1}\dfrac{\alpha}{x^{\alpha+1}},\;x>n^{-1/\alpha}.$ Note that $F$ satisfies (\ref{Introduction_e1}) with $a_n=n^{\frac{1}{\alpha}},$ $b_n=0$ and $G=\Phi_\alpha,$ the Fr\'{e}chet law with exponent $\alpha.$ From (\ref{Introduction_e3}), we have $ \lim_{n \rightarrow \infty} H(g_n) = 1 - \lim_{n \rightarrow \infty} I_2(n),$  where \begin{eqnarray*}
	I_2(n) & = & \int_{n^{-1/\alpha}}^{\infty}\left(1-\frac{1}{nx^\alpha}\right)^{n-1}\frac{\alpha}{x^{\alpha+1}}\log\alpha \,dx, \\
&& -(\alpha+1)\int_{n^{-1/\alpha}}^{\infty}\left(1-\frac{1}{nx^\alpha}\right)^{n-1}\frac{\alpha}{x^{\alpha+1}}\log x \,dx, \\
	&=& I_{A}(n)+I_{B}(n),\;\;\;\mbox{say.}
\end{eqnarray*}
	We have
\begin{eqnarray}\label{pareto_I2}
	I_{A}(n) &=& \log\alpha\int_{n^{-1/\alpha}}^{\infty}\left(1-\frac{1}{nx^\alpha}\right)^{n-1}\frac{\alpha}{x^{\alpha+1}} \,dx = \log\alpha.
\end{eqnarray}
And $I_{B}(n)=-(\alpha+1)\int_{n^{-1/\alpha}}^{\infty}\left(1-\frac{1}{nx^\alpha}\right)^{n-1}\frac{\alpha}{x^{\alpha+1}}\log x \,dx.$
	Putting $u=\dfrac{1}{x^\alpha},$ we get $du=-\dfrac{\alpha}{x^{\alpha+1}}dx,$ and $\log x=-\dfrac{\log u}{\alpha}$ so that
$I_{B}(n)= \frac{\alpha+1}{\alpha}\int_{0}^{n}\left(1-\frac{u}{n}\right)^{n-1}\log u \,du.$ Again putting $u=e^{-v}$, $du=-e^{-v}\,dv$ and we get
\begin{eqnarray}\label{pareto_I3}
I_{B}(n)&=&-\dfrac{(\alpha+1)}{\alpha}\int_{-\log n}^{\infty}v\left(1-\frac{e^{-v}}{n}\right)^{n-1}e^{-v}dv,\nonumber\\
&=&-\dfrac{(\alpha+1)}{\alpha}\int_{-\log n}^{\infty}v \,d\left(1-\frac{e^{-v}}{n}\right)^{n}.
\end{eqnarray}
Now, using Lemma \ref{Lemma2}, we get
\begin{equation}\label{pareto_I4}
	\lim_{n\to\infty}I_{B}(n)=-\dfrac{(\alpha+1)}{\alpha}\lim_{n\to\infty}E\left(Z_n-\log n\right)=-\dfrac{(\alpha+1)}{\alpha}\gamma.
\end{equation}
	Therefore, from (\ref{pareto_I2}) and (\ref{pareto_I4}), we get $\lim_{n\to\infty}H(g_n) = 1-\log\alpha +\frac{(\alpha+1)}{\alpha}\gamma, = H(\phi_\alpha).$ 

Now, for proving relative entropy convergence, we have
\begin{eqnarray}
	\Delta_{\phi_\alpha}(g_n)&=&-\int_{n^{-1/\alpha}}^{\infty}g_n(x)\log \phi_\alpha(x)\,dx,\nonumber\\
		&=&-\int_{n^{-1/\alpha}}^{\infty}\left(1-\dfrac{1}{nx^\alpha}\right)^{n-1}\dfrac{\alpha}{x^{\alpha+1}}\log(\alpha x^{-\alpha-1}e^{-x^{-\alpha}})\,dx,\nonumber\\
	&=&-\int_{n^{-1/\alpha}}^{\infty}\left(1-\dfrac{1}{nx^\alpha}\right)^{n-1}\dfrac{\alpha}{x^{\alpha+1}}\log\alpha \,dx\nonumber\\
	&&+\int_{n^{-1/\alpha}}^{\infty}\left(1-\dfrac{1}{nx^\alpha}\right)^{n-1}\dfrac{\alpha(\alpha+1)}{x^{\alpha+1}}\log x \, dx+\int_{n^{-1/\alpha}}^{\infty}\left(1-\dfrac{1}{nx^\alpha}\right)^{n-1}\dfrac{\alpha}{x^{\alpha+1}}x^{-\alpha}\, dx,\nonumber\\
	&=&I_{C}(n)+I_{D}(n)+I_{E}(n).\nonumber
\end{eqnarray}
Here
\begin{eqnarray}\label{maxpareto_I1}
	I_{C}(n)&=&-\int_{n^{-1/\alpha}}^{\infty}\left(1-\dfrac{1}{nx^\alpha}\right)^{n-1}\dfrac{\alpha}{x^{\alpha+1}}\log\alpha \, dx = -\log\alpha. \\
I_{D}(n) & = & \int_{n^{-1/\alpha}}^{\infty}\left(1-\dfrac{1}{nx^\alpha}\right)^{n-1}\dfrac{\alpha(\alpha+1)}{x^{\alpha+1}}\log x \,dx. \nonumber
\end{eqnarray}
	Putting $u=\alpha \log x,$ $du=\dfrac{\alpha}{x}\, dx,$ and
\[I_{D}(n)=\dfrac{\alpha+1}{\alpha}\int_{-\log n}^{\infty}u\left(1-\dfrac{e^{-u}}{n}\right)^{n-1}e^{-u}\, du=\dfrac{\alpha+1}{\alpha}\int_{-\log n}^{\infty}u \, d\left(1-\dfrac{e^{-u}}{n}\right)^{n}.\]
	Using Lemma \ref{Lemma2},
\begin{eqnarray}\label{maxpareto_I21}
	\lim_{n\to\infty}I_{D}(n)=\dfrac{\alpha+1}{\alpha}\lim_{n\to\infty}E(Z_n-\log n)=\dfrac{\alpha+1}{\alpha}\gamma.
\end{eqnarray}
Finally $I_{E}(n)=\int_{n^{-1/\alpha}}^{\infty}\left(1-\dfrac{1}{nx^\alpha}\right)^{n-1}\dfrac{\alpha}{x^{\alpha+1}}x^{-\alpha}dx.$
Putting $u=\dfrac{x^{-\alpha}}{n},$ $du=-\alpha x^{-\alpha-1}\dfrac{dx}{n},$ and
\begin{eqnarray}\label{Maxpareto_I3}
I_{E}(n)&=&-\int_{0}^{1}n^2u\left(1-u\right)^{n-1}du=-\int_{0}^{1}nu \, d\left(1-u\right)^{n},\nonumber\\
& = & \dfrac{n}{n+1}.
\end{eqnarray}
	From (\ref{maxpareto_I1}), (\ref{maxpareto_I21}) and (\ref{Maxpareto_I3}), we have $\lim_{n\to\infty}\Delta_{\phi_\alpha}(g_n)=-\log\alpha+\dfrac{\alpha+1}{\alpha}\gamma+1 = H(\phi_\alpha).$ Therefore, from (\ref{Introduction_e4}), $\lim_{n\to\infty}D(g_n\|\phi_\alpha)=0.$
\end{proof}
\begin{proof}[\textbf{Proof of Theorem \ref{Thmexm_ent} (b)}] In this case, we have $g_n(x)=\left(1+\frac{x}{n}\right)^{n-1}, -n<x<0;$ and $F$ satisfies (\ref{Introduction_e1}) with $a_n={\dfrac{1}{n}},\;b_n=1,$ and $G=\Psi_1,$ the Weibull law. Then, by (\ref{Introduction_e3}), $ \lim_{n\to\infty}H(g_n)= 1-\lim_{n\to\infty}I_2(n),$
	where, $I_2(n)=0$ since $\log(na_nf(a_nx+b_n))=\log 1=0.$
	Since entropy of Weibull law with $\alpha=1$ is $1,$ we get $\lim_{n\rightarrow\infty}H(g_n)=H(\psi_1),$ and by Lemma \ref{Lemma1}, $H(g_n)$ increases to $H(\psi_1)$ with $n.$
	
	We have $\; \Delta_{\psi_1}(g_n)=-\int_{-n}^{0}g_n(x)\log \psi_1(x)\,dx $ $ =-\int_{-n}^{0}\left(1+\frac{x}{n}\right)^{n-1}x \,dx.$	Using Theorem \ref{Moment}-(ii), we have $ \lim_{n\to\infty}\Delta_{\psi_1}(g_n)=H(\psi_1)=1$ and hence $\lim_{n\to\infty}D(g_n\|\psi_1)=0.$
\end{proof}
\begin{proof}[\textbf{Proof of Theorem \ref{Thmexm_ent} (c) }] We have $g_n(x)=ne^{-(x+\log n)}\left(1-e^{-(x+\log n)}\right)^{n-1},\;\;-\log n<x.$ 	 Note that $F$ satisfies (\ref{Introduction_e1}) with $a_n=1,$ $b_n=\log n$ and $G=\Lambda,$ the Gumbel law. Then by (\ref{Introduction_e3}),
$\lim_{n\to\infty}H(g_n) = 1-\lim_{n\to\infty}I_2(n), $
	where 
\begin{eqnarray}\label{exponential_I2}
 I_2(n)&=& \int_{-\log n}^{\infty}ne^{-(x+\log n)}\left(1-e^{-(x+\log n)}\right)^{n-1}\log\left(ne^{-(x+\log n)}\right)dx,\nonumber\\
&=& -\int_{-\log n}^{\infty}xd\left(1-\dfrac{e^{-x}}{n}\right)^n.
\end{eqnarray}
	By Lemma \ref{Lemma2}, $\lim_{n\to\infty}I_2(n)=-\gamma,$ and hence $\lim_{n\to\infty}H(g_n)=1+\gamma=H(\lambda).$

From (\ref{Introduction_e4}),
\begin{eqnarray}
	\Delta_{\lambda}(g_n)&=&-\int_{-\log n}^{\infty}g_n(x)\log\lambda(x) \, dx,\nonumber\\
	&=&-\int_{-\log n}^{\infty}ne^{-(x+\log n)}\left(1-e^{-(x+\log n)}\right)^{n-1}\log\left(e^{-x}e^{-e^{-x}}\right)dx,\nonumber\\
	&=&\int_{-\log n}^{\infty}xne^{-(x+\log n)}\left(1-e^{-(x+\log n)}\right)^{n-1}dx,\nonumber\\
	&&+\int_{-\log n}^{\infty}ne^{-(x+\log n)}\left(1-e^{-(x+\log n)}\right)^{n-1}e^{-x} \,dx=I_{A}(n)+I_{B}(n).\nonumber
\end{eqnarray}
Here
\begin{eqnarray}
	I_{A}(n)&=&\int_{-\log n}^{\infty}xn e^{-(x+\log n)}\left(1-e^{-(x+\log n)}\right)^{n-1}dx.
\end{eqnarray}
Using Theorem \ref{Moment}-(iii), $\lim_{n\to\infty}I_{A}(n)=\lim_{n\to\infty}E(Z_n)=\gamma.\;$ Next \\ $I_{B}(n)$ $=\int_{-\log n}^{\infty}e^{-x} \,d\left(1-\dfrac{e^{-x}}{n}\right)^{n}.$ Taking $e^{-x}=u$ we have $I_{B}(n)=-\int_{0}^{n}u \, d\left(1-\dfrac{u}{n}\right)^{n}.$
From Theorem \ref{Moment}-(ii), $\lim_{n\to\infty}I_{B}(n)=1.$ Therefore $\lim_{n\to\infty}\Delta_{\lambda}(g_n)=1+\gamma=H(\lambda),$ so that $\lim_{n\to\infty}D(g_n\|\lambda)=0.$
\end{proof}
\begin{proof}[\textbf{Proof of Theorem \ref{Thmexm_ent} (d)}]
	The pdf of $F$ is $f(x)=\dfrac{1}{\sqrt{2\pi}}e^{-\frac{x^2}{2}},\;\;x\in\Real.$ Hence,
\begin{eqnarray}\label{normal_g}
g_n(x)&=&\dfrac{na_n}{\sqrt{2\pi}}e^{-\frac{(a_nx+b_n)^2}{2}}\left(\int_{-\infty}^{a_nx+b_n}\dfrac{1}{\sqrt{2\pi}}e^{-\frac{v^2}{2}}dv\right)^{n-1}, x \in R.
\end{eqnarray}
	Note that $F$ satisfies (\ref{Introduction_e1}) with $b_n=\sqrt{2\log n}-\dfrac{\log\log n+\log(4\pi)}{2\sqrt{2\log n}},\;\;a_n=\dfrac{1}{\sqrt{2\log n}},$
and $G=\Lambda,$ the Gumbel law. By (\ref{Introduction_e3}),
\begin{eqnarray}
	\lim_{n\to\infty}H(g_n(x))&=&1-\lim_{n\to\infty}I_2(n),\nonumber
\end{eqnarray}
where $I_2(n)=\int_{-\infty}^{\infty}\log\left(\dfrac{na_n}{\sqrt{2\pi}}e^{-\frac{(a_nx+b_n)^2}{2}}\right)dF^n(a_nx+b_n).$ Now,
\begin{eqnarray}\label{normal_I1}
	(a_nx+b_n)^2 &=&\left(\dfrac{x}{\sqrt{2\log n}}\right)^2+\left(\sqrt{2\log n}-\dfrac{\log\log n+\log(4\pi)}{2\sqrt{2\log n}}\right)^2\nonumber\\
	&&+2\left(\dfrac{x}{\sqrt{2\log n}}\left(\sqrt{2\log n}-\dfrac{\log\log n+\log(4\pi)}{2\sqrt{2\log n}}\right)\right),\nonumber\\
	&=&\dfrac{x^2}{2\log n}+\left(2\log n+\dfrac{\log\log n+\log(4\pi))^2}{8\log n}-(\log\log n+\log(4\pi))\right)\nonumber\\
	&&+2x\left(1-\dfrac{\log\log n+\log(4\pi)}{4\log n}\right),\nonumber\\
	&=&\dfrac{x^2}{2\log n}+\log\left(\dfrac{n^2}{4\pi\log n}\right)+o_1(n)+2x(1-o_2(n)),\nonumber
\end{eqnarray}
	where $o_1(n)=\dfrac{(\log\log n+\log(4\pi))^2}{8\log n}$ and $o_2(n)=\dfrac{\log\log n+\log(4\pi)}{4\log n},$ and $o_1(n),\;o_2(n)$ tend to $0$ as $n\to\infty.$ Therefore,
\begin{eqnarray}\label{normal_I11}
\exp\Big\{-\frac{(a_nx+b_n)^2}{2}\Big\}&=&\exp\Big\{-\frac{1}{2}\left(\dfrac{x^2}{2\log n}+\log\left(\dfrac{n^2}{4\pi\log n}\right)+o_1(n)+2x(1-o_2(n))\right)\Big\},\nonumber\\
 &=&\dfrac{\sqrt{4\pi\log n}}{n}\exp\Big\{-\dfrac{x^2}{4\log n}-\dfrac{o_1(n)}{2}-x(1-o_2(n))\Big\},
\end{eqnarray}
From (\ref{normal_g}) and (\ref{normal_I11}),
\begin{eqnarray}
g_n(x)&=&\exp\Big\{-\dfrac{x^2}{4\log n}-\dfrac{o_1(n)}{2}-x(1-o_2(n))\Big\}\nonumber\\
&&F^{n-1}\left(\dfrac{2x+4\log n-\log\log n+\log(4\pi)}{2\sqrt{2\log n}}\right).
\end{eqnarray}
	Hence
\begin{eqnarray}\label{normal_I2}
	I_2(n)&=&\int_{-\infty}^{\infty}\log\left(\dfrac{na_n}{\sqrt{2\pi}}e^{-\frac{(a_nx+b_n)^2}{2}}\right)dF^n(a_nx+b_n),\nonumber\\
	&=&\int_{-\infty}^{\infty}\log\left(\dfrac{n\sqrt{4\pi\log n}}{n\sqrt{4\pi\log n}}\exp\Big\{-\dfrac{x^2}{4\log n}-\dfrac{o_1(n)}{2}-x(1-o_2(n))\Big\}\right)dF^n(a_nx+b_n),\nonumber\\
	&=&\int_{-\infty}^{\infty}\log\left(\exp\Big\{-\dfrac{x^2}{4\log n}-\dfrac{o_1(n)}{2}-x(1-o_2(n))\Big\}\right)dF^n(a_nx+b_n),\nonumber\\
	&=&-\int_{-\infty}^{\infty}\dfrac{x^2}{4\log n}dF^n(a_nx+b_n)-\int_{-\infty}^{\infty}\dfrac{o_1(n)}{2}dF^n(a_nx+b_n),\nonumber\\
	&&-\int_{-\infty}^{\infty}x(1-o_2(n))dF^{n}(a_nx+b_n)=I_{A}(n)+I_{B}(n)+I_{C}(n), \;\;\text{say,} \nonumber
\end{eqnarray}
	with $I_{A}(n)=-\dfrac{o_1(n)}{2}\rightarrow 0$ as $n\to\infty,$
\[I_{B}(n)=-\int_{-\infty}^{\infty}\left(\dfrac{x^2}{4\log n}\right)dF^n(a_nx+b_n)\longrightarrow 0,\;\;as\;\;n\to\infty,\]
since $\lim_{n\to\infty}\int_{-\infty}^{\infty}x^2dF^n(a_nx+b_n)=\Gamma^{(2)}(1),$ from Theorem~\ref{Moment}-(iii), and
\[I_{C}(n)=-\int_{-\infty}^{\infty}x(1-o_2(n))dF^n(a_nx+b_n)\longrightarrow -\gamma,\;\; as\;\;n\to\infty,\]
	by Theorem \ref{Moment}-(iii).
	Hence $\lim_{n\to\infty}H(g_n)=1+\gamma=H(\lambda).$
	
	We have
\begin{eqnarray}\label{normal_delta}
	\Delta_{\lambda}(g_n)&=&-\int_{-\infty}^{\infty}g_n(x)\log\lambda(x) \, dx,\nonumber\\
	&=&-\int_{-\infty}^{\infty}na_n\dfrac{1}{\sqrt{2\pi}}e^{-\frac{(a_nx+b_n)^2}{2}}F^{(n-1)}(a_nx+b_n)\log(e^{-x}e^{-e^{-x}}) \, dx,\nonumber\\
	&=&\int_{-\infty}^{\infty}xna_n\dfrac{1}{\sqrt{2\pi}}e^{-\frac{(a_nx+b_n)^2}{2}}F^{(n-1)}(a_nx+b_n) \, dx,\nonumber\\
	&&+\int_{-\infty}^{\infty}na_n\dfrac{1}{\sqrt{2\pi}}e^{-\frac{(a_nx+b_n)^2}{2}}F^{(n-1)}(a_nx+b_n)e^{-x} \, dx,\nonumber\\
	&=&I_{A}(n)+I_{B}(n), \; \text{say.}
\end{eqnarray}
Here $I_{A}(n)=\int_{-\infty}^{\infty}xdF^n(a_nx+b_n)=E(Z_n),$ and using Theorem~\ref{Moment}-(iii),
\begin{eqnarray}\label{normal_delta_I3}
\lim_{n\to\infty}I_{A}(n)=\lim_{n\to\infty}E(Z_n)=\gamma.
\end{eqnarray}
And, $I_{B}(n)=\int_{-\infty}^{\infty}e^{-x}dF^n(a_nx+b_n).$ Using integration by parts, we have $ I_{B}(n)$ $=\int_{-\infty}^{\infty}e^{-x}F^n(a_nx+b_n) \, dx.$ So $\lim_{n\to\infty}F^n(a_nx+b_n)=\Lambda(x),$
and for large $n$ and $x\in [-L,L],$ with $L>0$
\[\abs{F^n(a_nx+b_n)-\Lambda(x)}<1\Leftrightarrow -1+\Lambda(x)<F^n(a_nx+b_n)<1+\Lambda(x).\]
Since $\int_{-L}^{L}e^{-x}(\Lambda(x)+ 1) \, dx<\infty,$ by the dominated convergence theorem (DCT), for $L>0,$
\begin{eqnarray}
\lim_{n\to\infty}\int_{-L}^{L}e^{-x}F^n(a_nx+b_n)dx&=&\int_{-L}^{L}e^{-x}\lim_{n\to\infty}F^n(a_nx+b_n)dx,\nonumber\\
&=&\int_{-L}^{L}\lambda(x)dx=\Lambda(L)-\Lambda(-L).\nonumber
\end{eqnarray}
Therefore
\begin{eqnarray}\label{normal_delta_I4}
\lim_{n\to\infty}I_{B}(n)&=&\lim_{L\to\infty}\lim_{n\to\infty}\int_{-L}^{L}e^{-x}F^n(a_nx+b_n)dx,\nonumber\\
&=&\lim_{L\to\infty}\Lambda(L)-\Lambda(-L)=1.
\end{eqnarray}
From (\ref{normal_delta}), (\ref{normal_delta_I3}) and (\ref{normal_delta_I4}), we have $\lim_{n\to\infty}\Delta_{\lambda}(g_n)=1+\gamma=H(\lambda),$ so that \\ $	\lim_{n\to\infty}D(g_n\|\lambda)=0.$
\end{proof}

\begin{proof}[\textbf{Proof of Theorem \ref{TPhi} (a)}] Since $F \in \mathcal{D}(\Phi_\alpha),$ from Proposition 1.11 in Resnick (1987), $1 - F$ is regularly varying so that
\begin{eqnarray} \label{RegVar}
\lim_{t \rightarrow \infty} \frac{\overline{F}(tx)}{\overline{F}(t)} & = &  x^{- \alpha}, \;\; x > 0; \;\;\;\; \mbox{and} \\
\lim_{n \rightarrow \infty} F^n(a_n x) & = & \Phi_{\alpha}(x), \; x \in \Real,  \label{MaxDomPhi}
\end{eqnarray}
with $\; a_n = F^{-1}(1 - \frac{1}{n}) = \inf\{x: F(x) > 1 - \frac{1}{n}\}, \, n \geq 1.\;$
Further, since $f$ is eventually nonincreasing, from Proposition 1.15 in Resnick (1987), $F$ satisfies the von Mises condition (Theorem (\ref{thm_von})):
\begin{equation} \label{vonMises}
\lim_{x \rightarrow \infty} \frac{x f(x)}{1-F(x)} = \alpha.
\end{equation}
Now, by Proposition 2.5(a) in Resnick (1987) (Theorem \ref{gn_conv}), (\ref{vonMises}) implies the following density convergence on compact sets:
\begin{equation} \label{DenConv_frechet}
\lim_{n \rightarrow \infty} g_n(x) = \phi_{\alpha}(x), \; x \in K \subset\Real,
\end{equation} where $\; K \;$ is a compact set, and $\;g_n\;$ is as in (\ref{Introduction_e2}) with $\;b_n = 0.\;$
From (\ref{1.5}), we have,
$$ \lim_{n \rightarrow \infty} H(g_n) = 1 - \lim_{n \rightarrow \infty} I_2(n), $$ where
\begin{eqnarray}
	 \lim_{n \rightarrow \infty} I_2(n)&=& \lim_{n \rightarrow \infty} \int_{\epsilon_n}^{\infty}\log(a_nnf(a_nx))g_n(x)dx,\;\;\text{with} \;0 < \epsilon_n=\dfrac{l(F)}{a_n} \rightarrow 0 \; \text{as} \;  n \rightarrow \infty,\nonumber\\
 &=& \lim_{n \rightarrow \infty} \int_{\epsilon_n}^{\infty}\log\left(\dfrac{na_nxf(a_nx)\overline{F}(a_nx)}{x\overline{F}(a_nx)}
 \right) g_n(x)dx,\nonumber 
\end{eqnarray}
Now, for constants $\;0<L'<L,\;$ we have
\begin{eqnarray}
\int_{L'}^{L}\log\left(\dfrac{na_nxf(a_nx)\overline{F}(a_nx)}{x\overline{F}(a_nx)}\right)g_n(x)dx &=&\int_{L'}^{L}\log\left(\dfrac{a_nxf(a_nx)}{\overline{F}(a_nx)}\right)g_n(x)dx \nonumber\\
&&+\int_{L'}^{L}\log\left(\dfrac{n\overline{F}(a_nx)}{x}\right)g_n(x)dx \nonumber\\
	&=&I_{A}(n)+I_{B}(n), \;\mbox{say.}\nonumber
\end{eqnarray}
Here $\;I_{A}(n)=\int_{L'}^{L}\log\left(\dfrac{a_nxf(a_nx)}{\overline{F}(a_nx)}\right)g_n(x)dx.\;$ From (\ref{vonMises}) and (\ref{DenConv_frechet}), it follows that
\[\lim_{n\to\infty}\log\left(\dfrac{a_nxf(a_nx)}{\overline{F}(a_nx)}\right)na_nf(a_nx)F^{n-1}(a_nx)=\phi_\alpha(x)\log\alpha,\]
for large $n$ and $x \in [L', L],$ so that
\begin{eqnarray}
&&\abs{\log\left(\dfrac{a_nxf(a_nx)}{\overline{F}(a_nx)}\right)na_nf(a_nx)F^{n-1}(a_nx)-\phi_\alpha(x)\log\alpha}< 1, \mbox{which is equivalent to } \nonumber\\
&& -1+\phi_\alpha(x)\log\alpha<\log\left(\dfrac{a_nxf(a_nx)}{\overline{F}(a_nx)}\right)na_nf(a_nx)F^{n-1}(a_nx)<1+\phi_\alpha(x)\log\alpha.
\end{eqnarray}
Since $\;\int_{L'}^{L}(\phi_\alpha(x)\log\alpha+ 1) dx<\infty,\;$
by the DCT,
\begin{eqnarray}\label{EntropyFrechet_e1}
\lim_{L'\to 0}\lim_{L\to\infty} \lim_{n\to\infty}I_{A}(n)&=&\lim_{L'\to 0}\lim_{L\to\infty}\lim_{n\to\infty}\int_{L'}^{L}\log\left(\dfrac{a_nxf(a_nx)}{\overline{F}(a_nx)}\right)na_nf(a_nx)F^{n-1}(a_nx)dx,\nonumber\\
&=&\int_{0}^{\infty}\lim_{n\to\infty}\log\left(\dfrac{a_nxf(a_nx)}{\overline{F}(a_nx)}\right)na_nf(a_nx)F^{n-1}(a_nx)dx,\nonumber\\
&=&\int_{0}^{\infty}\log\alpha\phi_\alpha(x)dx=\log\alpha.
\end{eqnarray}
Next, $\; I_{B}(n) = \int_{L'}^{L}\log\left(\dfrac{n\overline{F}(a_nx)}{x}\right)na_nf(a_nx)F^{n-1}(a_nx)dx.\;$ From (\ref{RegVar}) and (\ref{DenConv_frechet}), we have
\[\lim_{n\to\infty}\log\left(\dfrac{n\overline{F}(a_nx)}{x}\right)na_nf(a_nx)F^{n-1}(a_nx)=-{(\alpha+1)}\log x\phi_{\alpha}(x),\]
and for large $n$ and $x\in [L', L],$
\begin{eqnarray}
&&\abs{\log\left(\dfrac{n\overline{F}(a_nx)}{x}\right)na_nf(a_nx)F^{n-1}(a_nx)+(\alpha+1)\phi_\alpha(x)\log x}< 1 \;\mbox{which is equivalent to}\nonumber\\
&& -1-(\alpha+1)\phi_\alpha(x)\log x<\log\left(\dfrac{n\overline{F}(a_nx)}{x}\right)na_nf(a_nx)F^{n-1}(a_nx)<1-(\alpha+1)\phi_\alpha(x)\log x.\nonumber
\end{eqnarray}
Since $\;\int_{L'}^{L}(1-(\alpha+1)\phi_\alpha(x)\log x)dx<\infty,\;$ by the DCT,
\begin{eqnarray}
\lim_{L'\to 0}\lim_{L\to\infty}\lim_{n\to\infty} I_{B}(n)&=&\lim_{L'\to 0}\lim_{L\to\infty}\lim_{n\to\infty}\int_{L'}^{L}\log\left(\dfrac{n\overline{F}(a_nx)}{x}\right)na_nf(a_nx)F^{n-1}(a_nx)dx,\nonumber\\
&=&\int_{0}^{\infty}\lim_{n\to\infty}\log\left(\dfrac{n\overline{F}(a_nx)}{x}\right)na_nf(a_nx)F^{n-1}(a_nx)dx,\nonumber\\
&=&\int_{0}^{\infty}\log(x^{-\alpha-1})\phi_\alpha(x)dx=-(\alpha+1)\int_{0}^{\infty}\log x\phi_\alpha(x)dx. \nonumber
\end{eqnarray}
Substituting \;$x^{-\alpha}=u,\;$ we get $\;-\alpha x^{-\alpha-1}dx=du,\;$ and
\begin{eqnarray}\label{EntropyFrechet_e2}
\lim_{L'\to 0}\lim_{L\to\infty}\lim_{n\to\infty} I_{B}(n)&=&\dfrac{\alpha+1}{\alpha}\int_{0}^{\infty}\log u e^{-u}du =-\dfrac{\alpha+1}{\alpha}\gamma.
\end{eqnarray}

	From (\ref{EntropyFrechet_e1}) and (\ref{EntropyFrechet_e2}),  we therefore have
\begin{eqnarray}
\lim_{n\to\infty}H(g_n)&=& 1 - \lim_{n\to\infty} I_2(n) = 1-\log\alpha+\dfrac{\alpha+1}{\alpha}\gamma=H(\phi_{\alpha}),
\end{eqnarray} completing the proof of the first part.

From (\ref{Introduction_e4}), we have,
\begin{eqnarray}
0\leq D(g_n\|\phi_\alpha)=-H(g_n)+\Delta_{\phi_\alpha}(g_n).\nonumber
\end{eqnarray}
where $\Delta_{\phi_\alpha}(g_n)=-\int_{\epsilon_n}^{\infty}\log\phi_\alpha(x)g_n(x) dx$ and $supp(g_n)\subseteq supp(\phi_\alpha).$

From (\ref{DenConv_frechet})
\[\lim_{n\to\infty}g_n(x)\log\phi_\alpha(x)=\phi_\alpha(x)\log\phi_\alpha(x),\]
 for large $n$ and $x\in [L',L]$ with $L',L>0,$
\begin{eqnarray}
&&\abs{g_n(x)\log \phi_\alpha(x)- \phi_\alpha(x)\log\phi_\alpha(x)}< 1,\nonumber\\
&&\Leftrightarrow -1+\phi_\alpha(x)\log\phi_\alpha(x)<g_n(x)\log \phi_\alpha(x)<1+ \phi_\alpha(x)\log\phi_\alpha(x)
\end{eqnarray}
Since
\[\int_{L'}^{L}(\phi_\alpha(x)\log\phi_\alpha(x)+ 1)dx<\infty,\]
by the DCT,
\begin{eqnarray}\label{RelativeFrechet_e1}
\lim_{n\to\infty}\Delta_{\phi_\alpha}(g_n)&=&-\lim_{L'\to 0}\lim_{L\to\infty}\lim_{n\to\infty}\int_{L'}^{L}g_n(x)\log \phi_\alpha(x)dx,\nonumber\\
&=&-\int_{0}^{\infty}\lim_{n\to\infty}g_n(x)\log \phi_\alpha(x)dx,\nonumber\\
&=&-\int_{0}^{\infty}\phi_\alpha(x)\log\phi_\alpha(x)dx,\nonumber\\
&=&H(\phi_\alpha).
\end{eqnarray}
	From, Theorem \ref{TPhi}(i) and (\ref{RelativeFrechet_e1})
\[\lim_{n\to\infty}D(g_n\|\phi_\alpha)=-\lim_{n\to\infty}H(g_n)+\lim_{n\to\infty}\Delta_{\phi_\alpha}(g_n)=0.\]

For the last part of the proof, from (\ref{Introduction_e3}) we can choose $\;a_n\;$ to be increasing so that $\;0<\epsilon_{n}<\epsilon_{n-1}.\;$ We then have
\begin{eqnarray}\label{inc_Fe1}
H(g_n)-H(g_{n-1})&=&\int_{\epsilon_{n-1}}^{\infty}\log g_{n-1}(x)g_{n-1}(x)dx-\int_{\epsilon_n}^{\infty}\log g_n(x)g_n(x)dx,\nonumber\\
&=&\int_{\epsilon_{n-1}}^{\infty}\log g_{n-1}(x)g_{n-1}(x)dx, \nonumber\\
&&-\int_{\epsilon_n}^{\epsilon_{n-1}}\log g_n(x)g_{n}(x)dx-\int_{\epsilon_{n-1}}^{\infty}\log g_n(x)g_{n}(x)dx,\nonumber\\
&\geq&\int_{\epsilon_{n-1}}^{\infty}(\log g_{n-1}(x)g_{n-1}(x)-\log g_n(x)g_{n}(x))dx,\nonumber\\
&&-\int_{\epsilon_n}^{\epsilon_{n-1}}\log g_n(x)g_{n-1}(x)dx>0,\nonumber
\end{eqnarray}
where, $g_n(x),$ is nonincreasing for $n$ large and $g_{n-1}(x)=0$ for $x\in (\epsilon_n,\epsilon_{n-1}).$ Therefore, $H(g_n)$ is increasing in $n.$
\end{proof}

\begin{proof}[\textbf{Proof of Theorem \ref{TPhi} (b)}]
Since $F \in \mathcal{D}(\Psi_\alpha),$ from Proposition 1.13 in Resnick (1987), $r(F) < \infty,$ and $F^{*}(t) = 1 - F(r(F) - \frac{1}{t})$ is regularly varying so that
\begin{eqnarray} \label{RegVar_w}
\lim_{t \rightarrow \infty} \frac{\overline{F}(r(F)-\frac{1}{tx})}{\overline{F}(r(F)-\frac{1}{t})} & = &  (-x)^{\alpha}, \;\; x < 0; \;\;\;\; \mbox{and} \\
\lim_{n \rightarrow \infty} F^n(a_n x+b_n) & = & \Psi_{\alpha}(x), \; x \in\Real,  \label{MaxDomPsi}
\end{eqnarray}
with $\; a_n = r(F)-F^{-1}(1 - \frac{1}{n}) =r(F)-\inf\{x: F(x) > 1 - \frac{1}{n}\}$ and $b_n=r(F), \, n \geq 1.\;$
Further, since $f$ is nonincreasing near $r(F),$ from Proposition 1.15 in Resnick (1987), $F$ satisfies the von Mises condition (Theorem \ref{thm_von}):
\begin{equation} \label{vonMises_w}
\lim_{x \uparrow r(F)} \frac{(r(F)-x) f(x)}{\overline{F}(x)} = \alpha.
\end{equation}
Now, by Proposition 2.5 (b) in Resnick (1987) (reproduced as Theorem \ref{gn_conv}), (\ref{vonMises_w}) implies the following density convergence on compact sets:
\begin{equation} \label{DenConv}
\lim_{n \rightarrow \infty} g_n(x) = \psi_{\alpha}(x), \; x \in S \subset\Real,
\end{equation} where $\; S \;$ is a compact set, and $\;g_n\;$ is as in (\ref{Introduction_e2}) with $\;b_n = r(F).\;$
From (\ref{1.5}), we have
\begin{eqnarray} \lim_{n \rightarrow \infty} H(g_n)  & = & 1 - \lim_{n \rightarrow \infty} I_2(n), \;\; \text{where} \nonumber \\ \lim_{n \rightarrow \infty} I_2(n) & = & \lim_{n \rightarrow \infty} \int_{-\infty}^{0}\log(a_nnf(a_nx+b_n))dF^{n}(a_nx+b_n) \nonumber \\ & = & \lim_{L \rightarrow - \infty} \lim_{n \rightarrow \infty} \int_{L}^{0}\log(a_nnf(a_nx+b_n))dF^{n}(a_nx+b_n). \nonumber
\end{eqnarray}
 Now for constant $L<0,$ we have
\begin{eqnarray}
&&\int_{L}^{0}\log\left(\dfrac{n(-a_nx)f(a_nx+b_n)\overline{F}(a_nx+b_n)}{(-x)\overline{F}(a_nx+b_n)}\right)dF^{n}(a_nx+b_n) \nonumber\\
&=&\int_{L}^{0}\log\left(\dfrac{(-a_nx)f(a_nx+b_n)}{\overline{F}(a_nx+b_n)}\right)dF^{n}(a_nx+b_n)+\int_{L}^{0}\log\left(\dfrac{n\overline{F}(a_nx+b_n)}{(-x)}\right)dF^{n}(a_nx+b_n),\nonumber\\
	&=&I_{A}(n)+I_{B}(n), \mbox{say.}\nonumber
\end{eqnarray}
	We have $\;I_{A}(n)=\int_{L}^{0}\log\left(-\dfrac{a_nxf(a_nx+b_n)}{\overline{F}(a_nx+b_n)}\right)na_nf(a_nx+b_n)F^{n-1}(a_nx+b_n)dx\;$ and from (\ref{vonMises_w}) and (\ref{DenConv}), since $a_n x + b_n \rightarrow r(F)$ as $n \rightarrow \infty,$
\begin{eqnarray*} \lim_{n\to\infty}\log\left(-\dfrac{a_nxf(a_nx+b_n)}{\overline{F}(a_nx+b_n)}\right)na_nf(a_nx+b_n)F^{n-1}(a_nx+b_n) & = & \log\alpha\psi_\alpha(x),\;x<0. \end{eqnarray*}
So, for large $n$ and $x\in [L,0],$
$$\abs{\log\left(-\dfrac{a_nxf(a_nx+b_n)}{\overline{F}(a_nx+b_n)}\right)na_nf(a_nx+b_n)F^{n-1}(a_nx+b_n)-\psi_\alpha(x)\log\alpha} < 1$$ which is equivalent to
$$ -1+\psi_\alpha(x)\log\alpha<\log\left(-\dfrac{a_nxf(a_nx+b_n)}{\overline{F}(a_nx+b_n)}\right)na_nf(a_nx+b_n)F^{n-1}(a_nx+b_n)
<1+\psi_\alpha(x)\log\alpha.$$
Since $\;\int_{L}^{0}(\psi_\alpha(x)\log\alpha+ 1)dx<\infty,\;$ by the DCT,
\begin{eqnarray}\label{EntropyWeibull_e1}
\lim_{L\to -\infty}\lim_{n\to\infty}I_A(n)&=&\lim_{L\to -\infty} \lim_{n\to\infty}\int_{L}^{0}\log\left(-\dfrac{a_nxf(a_nx+b_n)}{\overline{F}(a_nx+b_n)}\right)na_nf(a_nx+b_n)F^{n-1}(a_nx+b_n)dx,\nonumber\\
&=&\int_{-\infty}^{0}\lim_{n\to\infty}\log\left(-\dfrac{a_nxf(a_nx+b_n)}{\overline{F}(a_nx+b_n)}\right)na_nf(a_nx+b_n)F^{n-1}(a_nx+b_n)dx,\nonumber\\
&=&\int_{-\infty}^{0}\log\alpha d\Psi_\alpha(x) =\log\alpha.
\end{eqnarray}
	Next, $\; I_{B}(n)=\int_{L}^{0}\log\left(-\dfrac{n\overline{F}(a_nx+b_n)}{x}\right)na_nf(a_nx+b_n)F^{n-1}(a_nx+b_n)dx.$
We have
\[\lim_{n\to\infty}\log\left(-\dfrac{n\overline{F}(a_nx+b_n)}{x}\right)na_nf(a_nx+b_n)F^{n-1}(a_nx+b_n)=(\alpha-1)\log(-x)\psi_\alpha(x).\]
By the Proposition 0.5 in Resnick (1987), $\;\lim_{n\to\infty}n\overline{F}(a_nx+b_n)=(-x)^{\alpha}, x < 0 \; $ and for large $n$ and $x\in[L,0]$ we have
$$\abs{\log\left(-\dfrac{n\overline{F}(a_nx+b_n)}{x}\right)na_nf(a_nx+b_n)F^{n-1}(a_nx+b_n)-\psi_\alpha(x)(\alpha-1)\log(-x)}<1,$$ $$\Leftrightarrow -1+\psi_\alpha(x)(\alpha-1)\log(-x)<\log\left(-\dfrac{n\overline{F}(a_nx+b_n)}{x}\right)na_nf(a_nx+b_n)F^{n-1}(a_nx+b_n)
<1+\psi_\alpha(x)(\alpha-1)\log (-x).$$
Since $\; \int_{L'}^{0}(\psi_\alpha(x)(\alpha-1)\log(-x)+ 1)dx<\infty,\;$ by the DCT,
\begin{eqnarray}
\lim_{L\to-\infty}\lim_{n\to\infty} I_B(n)&=&\lim_{L\to -\infty} \lim_{n\to\infty}\int_{L}^{0}\log\left(-\dfrac{n\overline{F}(a_nx+b_n)}{x}\right)na_nf(a_nx+b_n)F^{n-1}(a_nx+b_n)dx,\nonumber\\
&=&\int_{-\infty}^{0}\lim_{n\to\infty}\log\left(-\dfrac{n\overline{F}(a_nx+b_n)}{x}\right)na_nf(a_nx+b_n)F^{n-1}(a_nx+b_n)dx,\nonumber\\
&=&(\alpha-1)\int_{-\infty}^{0}\log(-x)\alpha (-x)^{\alpha-1}e^{-(-x)^{\alpha}}dx. \nonumber
\end{eqnarray}
	Putting $\; (-x)^{\alpha}=u,\;$ we have  $\;-\alpha (-x)^{\alpha-1}dx=du\;$ and
\begin{eqnarray}\label{EntropyWeibull_e2}
\lim_{L\to-\infty}\lim_{n\to\infty} I_B(n) =\dfrac{\alpha-1}{\alpha}\int_{0}^{\infty}\log ue^{-u}du&=&-\dfrac{\alpha-1}{\alpha}\gamma.
\end{eqnarray}
 From (\ref{EntropyWeibull_e1}) and (\ref{EntropyWeibull_e2}), we have
\[\lim_{n\to\infty}H(g_n)=1-\log\alpha+\dfrac{\alpha-1}{\alpha}\gamma=H(\psi_\alpha).\]

 From (\ref{Introduction_e4}), we have,
\begin{eqnarray}
0\leq D(g_n\|\psi_\alpha)=-H(g_n)+\Delta_{\psi_\alpha}(g_n).\nonumber
\end{eqnarray}
where, $\Delta_{\psi_\alpha}(g_n)=-\int_{-\infty}^{\nu_n}\log\psi_\alpha(x)dF^{n}(a_nx+b_n),$ and $supp(g_n)\subseteq supp(\psi_\alpha).$

From (\ref{DenConv})
\[\lim_{n\to\infty}g_n(x)\log \psi_\alpha(x)=\psi_\alpha(x)\log\psi_\alpha(x),\]
 for large $n$ and $x\in [L',L]$ with $L',L<0,$
\begin{eqnarray}
&&\abs{g_n(x)\log \psi_\alpha(x)-\psi_\alpha(x)\log\psi_\alpha(x)}< 1\nonumber\\
&&\Leftrightarrow -1+\psi_\alpha(x)\log\psi_\alpha(x)<g_n(x)\log \psi_\alpha(x)-\psi_\alpha(x)\log\psi_\alpha(x)<1+\psi_\alpha(x)\log\psi_\alpha(x).\nonumber
\end{eqnarray}
Since
\[\int_{L'}^{L}(\psi_\alpha(x)\log\psi_\alpha(x)+ 1)dx<\infty,\]
by the DCT,
\begin{eqnarray}\label{RelativeWeibull_e1}
\lim_{n\to\infty}\Delta_{\psi_\alpha}(g_n)&=&-\lim_{L'\to -\infty}\lim_{L\to 0}\lim_{n\to\infty}\int_{L'}^{L}g_n(x)\log \psi_\alpha(x)dx,\nonumber\\
&=&-\int_{-\infty}^{0}\lim_{n\to\infty}g_n(x)\log \psi_\alpha(x)dx,\nonumber\\
&=&-\int_{-\infty}^{0}\psi_\alpha(x)\log\psi_\alpha(x)dx,\nonumber\\
&=& H(\psi_\alpha).
\end{eqnarray}

From Theorem \ref{TPhi}(ii) and (\ref{RelativeWeibull_e1})
\[\lim_{n\to\infty}D(g_n\|\psi_\alpha)=-\lim_{n\to\infty}H(g_n)+\lim_{n\to\infty}\Delta_{\psi_\alpha}(g_n)=0.\]
For the last part of the proof, from (\ref{Introduction_e3}), we have
\begin{eqnarray}\label{inc_we1}
H(g_n)-H(g_{n-1})&=&\int_{-\infty}^{0}\log g_{n-1}(x)g_{n-1}(x)dx-\int_{-\infty}^{0}\log g_n(x)g_n(x)dx\geq 0,\nonumber
\end{eqnarray}
where, $g_n(x)$ is nonincreasing. Therefore, $H(g_n)$ is increasing.
\end{proof}
\begin{proof}[\textbf{Proof of Theorem \ref{TPhi} (c)}]
Since $F \in \mathcal{D}(\Lambda), \; r(F) \leq \infty,\;$ and from Proposition 1.1 in Resnick (1987), $1 - F$ is $\Gamma$ varying so that
\begin{eqnarray} \label{RegVar_g}
\lim_{t \rightarrow r(F)} \frac{\overline{F}(t+xu(t))}{\overline{F}(t)} & = &  e^{-x}, \;\; x\in\Real,
\end{eqnarray}
where the function $u(t)=\int_{t}^{r(F)}\overline{F}(s)ds/\overline{F}(t)$ is called an auxiliary function. We also have
\begin{eqnarray}
\lim_{n \rightarrow \infty} F^n(a_n x+b_n) & = & \Lambda(x), \; x \in\Real,  \label{MaxDomLambda}
\end{eqnarray}
with $\; a_n = u(b_n)$ and $b_n=F^{-1}(1 - \frac{1}{n}) = \inf\{x: F(x) > 1 - \frac{1}{n}\}, \, n \geq 1.\;$
Further, since $f$ is nonincreasing, from Proposition 1.17 in Resnick (1987), $F$ satisfies the von Mises condition (Theorem \ref{thm_von}):
\begin{equation} \label{vonMises_g}
\lim_{x \rightarrow r(F)} \dfrac{f(x)\int_{x}^{r(F)}\overline{F}(t)dt}{\overline{F}(x)^2} = 1.
\end{equation}
Now, by Proposition 2.5(c) in Resnick (1987) (Theorem \ref{gn_conv}), (\ref{vonMises_g}) implies the following density convergence on compact sets:
\begin{equation} \label{DenConv_g}
\lim_{n \rightarrow \infty} g_n(x) = \lambda(x), \; x \in S \subset \Real,
\end{equation} where $\; S \;$ is a compact set, and $\;g_n\;$ is as in (\ref{Introduction_e2}).
From (\ref{1.5}), we have
\begin{eqnarray} \lim_{n \rightarrow \infty} H(g_n) & = & 1 - \lim_{n \rightarrow \infty} I_2(n), \;\; \text{where} \nonumber \\
	I_2(n)&=&\int_{-\infty}^{\infty}\log(a_nnf(a_nx+b_n))dF^{n}(a_nx+b_n),\nonumber\\
 &=&\int_{-\infty}^{\infty}\log\left(\dfrac{nu(b_n)u(a_nx+b_n)f(a_nx+b_n)\overline{F}(a_nx+b_n)}{u(a_nx+b_n)
 \overline{F}(a_nx+b_n)}\right)dF^{n}(a_nx+b_n), \nonumber \\  & = & \lim_{L \rightarrow \infty} \int_{-L}^{L}\log\left(\dfrac{nu(b_n)u(a_nx+b_n)f(a_nx+b_n)\overline{F}(a_nx+b_n)}{u(a_nx+b_n)\overline{F}(a_nx+b_n)}\right)dF^{n}(a_nx+b_n). \nonumber
\end{eqnarray}
For constants $\;L>0,\;$ we then have,
\begin{eqnarray}
&&\int_{-L}^{L}\log\left(\dfrac{nu(b_n)u(a_nx+b_n)f(a_nx+b_n)\overline{F}(a_nx+b_n)}{u(a_nx+b_n)\overline{F}(a_nx+b_n)}\right)dF^{n}(a_nx+b_n)\nonumber\\
 &=&\int_{-L}^{L}\log\left(\dfrac{u(b_n)}{u(a_nx+b_n)}\right)dF^n(a_nx+b_n)+\int_{-L}^{L}\log\left(\dfrac{u(a_nx+b_n)f(a_nx+b_n)}{\overline{F}(a_nx+b_n)}\right)dF^{n}(a_nx+b_n) \nonumber\\
	&&+\int_{-L}^{L}\log(n\overline{F}(a_nx+b_n))dF^{n}(a_nx+b_n),\nonumber\\
	&=&I_{A}(n)+I_{B}(n)+I_{C}(n), \;\;\text{say.}\nonumber
\end{eqnarray}
	We have
\[I_{A}(n)=\int_{-L}^{L}\log\left(\dfrac{u(b_n)}{u(a_nx+b_n)}\right)a_nnf(a_nx+b_n)F^{n-1}(a_nx+b_n)dx.\]
	Using Theorem (\ref{lem2_appendix}),
\[\lim_{n\to\infty}\log\left(\dfrac{u(b_n)}{u(a_nx+b_n)}\right)a_nnf(a_nx+b_n)F^{n-1}(a_nx+b_n)=0\]
	locally uniformly in $\;x\in\Real.\;$ And for large $n$ and $x\in[-L,L],$
\begin{eqnarray}
&&\abs{\log\left(\dfrac{u(b_n)}{u(a_nx+b_n)}\right)a_nnf(a_nx+b_n)F^{n-1}(a_nx+b_n)}< 1,\nonumber\\
&&\Leftrightarrow -1<\log\left(\dfrac{u(b_n)}{u(a_nx+b_n)}\right)a_nnf(a_nx+b_n)F^{n-1}(a_nx+b_n)<1.\nonumber
\end{eqnarray}
Since $\;\int_{-L}^{L}dx< \infty,\;$ by the DCT, we have
\begin{eqnarray}\label{EntropyGumbel_e1}
 \lim_{L \rightarrow \infty} \lim_{n\to\infty}I_A(n)&=&- \lim_{L \rightarrow \infty} \lim_{n\to\infty}\int_{-L}^{L}\log\left(\dfrac{u(a_nx+b_n)}{u(b_n)}\right)a_nnf(a_nx+b_n)F^{n-1}(a_nx+b_n)dx,\nonumber\\
	&=&- \int_{-\infty}^{\infty}\lim_{n\to\infty}\log\left(\dfrac{u(a_nx+b_n)}{u(b_n)}\right)a_nnf(a_nx+b_n)F^{n-1}(a_nx+b_n)dx,\nonumber\\
	&=& 0.
\end{eqnarray}
 Next,
\begin{eqnarray}
	I_{B}(n)&=&\int_{-L}^{L}\log\left(\dfrac{u(a_nx+b_n)f(a_nx+b_n)}{\overline{F}(a_nx+b_n)}\right)a_nnf(a_nx+b_n)F^{n-1}(a_nx+b_n)dx. \nonumber
\end{eqnarray}
We have
\[\lim_{n\to\infty}\log\left(\dfrac{u(a_nx+b_n)f(a_nx+b_n)}{\overline{F}(a_nx+b_n)}\right)a_nnf(a_nx+b_n)F^{n-1}(a_nx+b_n)=0.\]
Hence by Theorem (\ref{lem3_appendix}), for large $n$ and $x\in[-L,L],$
\begin{eqnarray}
&&\abs{\log\left(\dfrac{u(a_nx+b_n)f(a_nx+b_n)}{\overline{F}(a_nx+b_n)}\right)a_nnf(a_nx+b_n)F^{n-1}(a_nx+b_n)}< 1\nonumber\\
&&\Leftrightarrow -1<\log\left(\dfrac{u(a_nx+b_n)f(a_nx+b_n)}{\overline{F}(a_nx+b_n)}\right)a_nnf(a_nx+b_n)F^{n-1}(a_nx+b_n)<1.\nonumber
\end{eqnarray}
Since $\;\int_{-L}^{L}dx<\infty,\;$ by the DCT, we have
\begin{eqnarray}\label{EntropyGumbel_e2}
  \lim_{L \rightarrow \infty} \lim_{n\to\infty}I_{B}(n)&=& \lim_{L \rightarrow \infty} \lim_{n\to\infty}\int_{-L}^{L}\log\left(\dfrac{u(a_nx+b_n)f(a_nx+b_n)}{\overline{F}(a_nx+b_n)}\right)\nonumber\\
 && a_nnf(a_nx+b_n)F^{n-1}(a_nx+b_n)dx,\nonumber\\
&=&\int_{-\infty}^{\infty}\lim_{n\to\infty}\log\left(\dfrac{u(a_nx+b_n)f(a_nx+b_n)}{\overline{F}(a_nx+b_n)}\right)a_nnf(a_nx+b_n)F^{n-1}(a_nx+b_n)dx,\nonumber\\
&=& 0.
\end{eqnarray}
	Finally,
\begin{eqnarray}
	I_{C}(n)&=&\int_{-L}^{L}\log(n\overline{F}(a_nx+b_n))a_nnf(a_nx+b_n)F^{n-1}(a_nx+b_n)dx.\nonumber
\end{eqnarray}
We have
\[\lim_{n\to\infty}\log(n\overline{F}(a_nx+b_n))a_nnf(a_nx+b_n)F^{n-1}(a_nx+b_n)=-x\lambda(x).\]
If $F\in\Gamma,$ then $n\overline{F}(a_nx+b_n)\to e^{-x}$ as $n\to\infty,$ and hence, for large $n$ and $x\in[-L,L],$
\begin{eqnarray}
&&\abs{\log(n\overline{F}(a_nx+b_n))a_nnf(a_nx+b_n)F^{n-1}(a_nx+b_n)+x\lambda(x)}< 1\nonumber\\
&&\Leftrightarrow -1+x\lambda(x)\log(n\overline{F}(a_nx+b_n))a_nnf(a_nx+b_n)F^{n-1}(a_nx+b_n)<1+x\lambda(x).\nonumber
\end{eqnarray}
Since $\;\int_{-L}^{L}(-x\lambda(x)+ 1)dx< \infty,\;$ by the DCT,  we have
\begin{eqnarray}\label{EntropyGumbel_e3}
 \lim_{L \rightarrow \infty} \lim_{n\to\infty}I_C(n)&=&  \lim_{L \rightarrow \infty} \lim_{n\to\infty}\int_{-L}^{L}a_nnf(a_nx+b_n)F^{n-1}(a_nx+b_n)\log(n\overline{F}(a_nx+b_n))dx,\nonumber\\
	&=&\int_{-\infty}^{\infty}\lim_{n\to\infty}a_nnf(a_nx+b_n)F^{n-1}(a_nx+b_n)\log(n\overline{F}(a_nx+b_n))dx,\nonumber\\
	&=&-\int_{-\infty}^{\infty}xe^{-x}e^{-e^{-x}}dx,\nonumber\\
	&=&\gamma.
\end{eqnarray}
	From (\ref{EntropyGumbel_e1}), (\ref{EntropyGumbel_e2}), (\ref{EntropyGumbel_e3}) we have,
\begin{eqnarray}
\lim_{n\to\infty}H(g_n)\,=\,1+\,\gamma\,=\,H(\lambda).\nonumber
\end{eqnarray}

From (\ref{Introduction_e4}), we have,
\begin{eqnarray}
0\leq D(g_n\|\lambda)=-H(g_n)+\Delta_{\lambda}(g_n).\nonumber
\end{eqnarray}
where, $\Delta_{\lambda}(g_n)=-\int_{-\infty}^{\infty}\log(\lambda(x))dF^{n}(a_nx+b_n),$ and $supp(g_n)\subseteq supp(\lambda).$ Therefore,

From (\ref{DenConv_g}) for large $n$ and $x\in[-L,L],$ for $L>0$
\[\lim_{n\to\infty}g_n(x)\log \lambda(x)=\lambda(x)\log\lambda(x),\]

\begin{eqnarray}
\abs{g_n(x)\log \lambda(x)-\lambda(x)\log\lambda(x)}< 1\
\Leftrightarrow -1+\lambda(x)\log\lambda(x)<g_n(x)\log \lambda(x)<1+\lambda(x)\log\lambda(x).\nonumber
\end{eqnarray}
Since
\[\int_{-L}^{L}(\lambda(x)\log\lambda(x)+ 1)dx<\infty,\]
by the DCT,
\begin{eqnarray}\label{RelativeGumbel_e1}
\lim_{n\to\infty}\Delta_{\lambda}(g_n)&=&-\lim_{L\to\infty}\lim_{n\to\infty}\int_{-L}^{L}g_n(x)\log \lambda(x)dx,\nonumber\\
&=&-\int_{-\infty}^{\infty}\lim_{n\to\infty}g_n(x)\log \lambda(x)dx,\nonumber\\
&=&-\int_{-\infty}^{\infty}\lambda(x)\log\lambda(x)dx=H(\lambda).
\end{eqnarray}
	From Theorem \ref{TPhi}(iii) and (\ref{RelativeGumbel_e1})
\[\lim_{n\to\infty}D(g_n\|\lambda)=-\lim_{n\to\infty}H(g_n)+\lim_{n\to\infty}\Delta_{\lambda}(g_n)=0.\]

For the last part of the proof, from (\ref{Introduction_e3}) we have
\begin{eqnarray}\label{inc_ge1}
H(g_n)-H(g_{n-1})&=&\int_{-\infty}^{\infty}\log g_{n-1}(x) g_{n-1}(x)dx-\int_{-\infty}^{\infty}\log g_n(x) g_n(x)dx\geq 0.\nonumber
\end{eqnarray}
where, $g_n(x)$ in nonincreasing. Therefore $H(g_n)$ is increasing.
\end{proof}

\begin{proof}[\textbf{Proof of Lemma \ref{Lemma.5}}]
The pdf $g_{k:n}$ of $G_{k:n}$ in (\ref{Introduction_e1.3}) is
\begin{eqnarray}
g_{k:n}(x)&=&\dfrac{n!a_nf(a_nx+b_n)(n\overline{F}(a_nx+b_n))^{k-1}F^{n-1}(a_nx+b_n)}{(k-1)!n^{k-1}(n-k)!F^{k-1}(a_nx+b_n)},\nonumber\\
&=&\dfrac{n!g_n(x)(n\overline{F}(a_nx+b_n))^{k-1}}{(k-1)!n^{k}(n-k)!F^{k-1}(a_nx+b_n)}, \nonumber
\end{eqnarray}
where $g_n(x)=na_nf(a_nx+b_n)F^{n-1}(a_nx+b_n)$ is the pdf of $F^n(a_nx+b_n).$ By Theorem \ref{gn_conv} and (\ref{Introduction_e1.2})
\begin{eqnarray}\label{order_density}
\lim_{n\to\infty}g_{k:n}(x)&=&\lim_{n\to\infty}\dfrac{n!g_n(x)(n\overline{F}(a_nx+b_n))^{k-1}}{(k-1)!n^{k}(n-k)!F^{k-1}(a_nx+b_n)},\nonumber\\
&=&g(x)\dfrac{(-\log G(x))^{k-1}}{(k-1)!} = K_k^{\prime}(x).
\end{eqnarray}
\end{proof}
\begin{proof}[\textbf{Proof of Lemma \ref{Lemma3}}]
For $k=1,$ we have 
\begin{eqnarray}
A(1)&=&\int_{0}^{\infty} e^{-u}\log u du=-\gamma.
\end{eqnarray}
For $k=2,$ we have 
\begin{eqnarray}
A(2) &=& \int_{0}^{\infty} ue^{-u}\log u du \nonumber \\ 
&=&\int_{0}^{\infty}\log ue^{-u}du+\int_{0}^{\infty}e^{-u}du, \text{using integration by parts,} \nonumber\\
&=&A(1)+\Gamma(1)=1-\gamma.\nonumber
\end{eqnarray}
Assuming the result for arbitrary $k-1,$ we have 
\begin{eqnarray}\label{lemma3_A}
A(k-1)&=& (k-2)!\left(-\gamma+\sum_{i=1}^{k-2}\dfrac{1}{i}\right),\nonumber\\
&=&-\gamma(k-2)!+\dfrac{(k-2)!}{k-2}+(k-2)\dfrac{(k-3)!}{k-3}+[(k-2)(k-3)]\dfrac{(k-4)!}{k-4}\nonumber\\
&&+\ldots+[(k-2)(k-3)(k-4)\ldots 5]\dfrac{4!}{4}+[(k-2)(k-3)\ldots 4]\dfrac{3!}{3}+(k-2)!,\nonumber\\
&=&-\gamma(k-2)!+\Gamma(k-2)+(k-2)\Gamma(k-3)+[(k-2)(k-3)]\Gamma(k-4),\nonumber\\
&&+\ldots+[(k-2)(k-3)(k-4)\ldots 5]\Gamma(4)+[(k-2)(k-3)\ldots 4]\Gamma(3)\nonumber\\
&&+[(k-2)(k-3)\ldots 3]\Gamma(2),\nonumber\\
&=&(k-2)A(k-2)+\Gamma(k-2).
\end{eqnarray}
We then have, 
\begin{eqnarray}
A(k)&=&\int_{0}^{\infty}u^{k-1}\log u e^{-u}du, \nonumber \\ 
&=&(k-1)\int_{0}^{\infty}u^{k-2}\log ue^{-u}du+\int_{0}^{\infty}u^{k-2}e^{-u}du,\nonumber\\
&=&(k-1)A(k-1)+\Gamma(k-1), \nonumber \\
&=&(k-1)!\left(-\gamma+\sum_{i=1}^{k-2}\dfrac{1}{i}\right)+\Gamma(k-1) \;\text{using}\; (\ref{lemma3_A}),\nonumber\\
&=&-\gamma(k-1)!+\dfrac{(k-1)!}{k-1}+(k-1)\dfrac{(k-2)!}{k-2}+[(k-1)(k-2)]\dfrac{(k-3)!}{k-3}\nonumber\\
&&+\ldots+[(k-1)(k-2)(k-3)\ldots 4]\dfrac{3!}{3}+[(k-1)(k-2)\ldots 3]\dfrac{2!}{2}+(k-1)!,\nonumber\\
&=&(k-1)!(-\gamma+\dfrac{1}{k-1}+\dfrac{1}{k-2}+\dfrac{1}{k-3}+\ldots+\dfrac{1}{3}+\dfrac{1}{2}+1),\nonumber\\
&=&(k-1)!\left(-\gamma+\sum_{i=1}^{k-1}\dfrac{1}{i}\right),\;k\geq 2.\nonumber
\end{eqnarray} Hence, by induction, the proof is complete. 
\end{proof}

\begin{proof}[\textbf{Proof of Lemma \ref{TPhi_orderEnt} (i)}] From (\ref{order_density}), $\;\phi_\alpha^{(k)}(x)=\dfrac{\alpha}{(k-1)!} x^{-\alpha k-1}e^{-x^{-\alpha}}, \;$ and hence the entropy of $\phi_\alpha^{(k)}(x)$ is
\begin{eqnarray}
H(\phi_\alpha^{(k)})&=&-\int_{0}^{\infty}\dfrac{\alpha}{(k-1)!} x^{-\alpha k-1}e^{-x^{-\alpha}}\log\left(\dfrac{\alpha}{(k-1)!} x^{-\alpha k-1}e^{-x^{-\alpha}}\right)dx.\nonumber
\end{eqnarray}
Putting $x^{-\alpha}=u,$ we have $-\alpha x^{-\alpha-1}dx=du,$ and 
\begin{eqnarray}
H(\phi_\alpha^{(k)}) &=&-\int_{0}^{\infty}\dfrac{1}{(k-1)!} u^{k-1}e^{-u}\log\left(\dfrac{\alpha}{(k-1)!} u^{\frac{\alpha k+1}{\alpha}}e^{-u}\right)du,\nonumber\\
&=&-\int_{0}^{\infty}\dfrac{1}{(k-1)!} u^{k-1}e^{-u}\log\dfrac{\alpha}{(k-1)!}du \nonumber\\
 &&-\dfrac{\alpha k+1}{\alpha(k-1)!}\int_{0}^{\infty}u^{k-1}e^{-u}\log u du +\int_{0}^{\infty}\dfrac{u^{k}e^{-u}}{(k-1)!}du, \nonumber\\
&=&-\log\dfrac{\alpha}{(k-1)!}+I_A+\dfrac{1}{(k-1)!}\Gamma(k+1),\label{ent_g1}
\end{eqnarray}
where 
\begin{eqnarray}
I_A&=&-\dfrac{\alpha k+1}{\alpha(k-1)!}\int_{0}^{\infty}u^{k-1}e^{-u}\log u du =-\dfrac{\alpha k+1}{\alpha(k-1)!}A(k), \nonumber \\ 
&=&-\dfrac{\alpha k+1}{\alpha}\left(-\gamma+\sum_{i=1}^{k-1}\dfrac{1}{i}\right), \label{order_frechet_I1}
\end{eqnarray}
using Lemma \ref{Lemma3}. From (\ref{ent_g1}) and (\ref{order_frechet_I1}), 
\begin{eqnarray}
H(\phi_\alpha^{(k)})=-\log\dfrac{\alpha}{(k-1)!}-\dfrac{\alpha k+1}{\alpha}\left(-\gamma+\sum_{i=1}^{k-1}\dfrac{1}{i}\right)+\dfrac{\Gamma(k+1)}{(k-1)!}.
\end{eqnarray}
\end{proof}
\begin{proof}[\textbf{Proof of Lemma \ref{TPhi_orderEnt} (ii)}] From (\ref{order_density}), $\;\psi_\alpha^{(k)}(x)=\dfrac{\alpha}{(k-1)!} (-x)^{\alpha k-1}e^{-(-x)^{\alpha}}\;$ and the entropy of $\psi_\alpha^{(k)}(x)$ is
\begin{eqnarray}
H(\psi_\alpha^{(k)})&=&-\int_{-\infty}^{0}\dfrac{\alpha}{(k-1)!} (-x)^{\alpha k-1}e^{-(-x)^{\alpha}}\log\left(\dfrac{\alpha}{(k-1)!} (-x)^{\alpha k-1}e^{-(-x)^{\alpha}}\right)dx.\nonumber
\end{eqnarray}
Putting $(-x)^{\alpha}=u,$ we have $-\alpha (-x)^{\alpha-1}dx=du$ and 
\begin{eqnarray}
H(\psi_\alpha^{(k)})&=&-\int_{0}^{\infty}\dfrac{1}{(k-1)!} u^{k-1}e^{-u}\log\left(\dfrac{\alpha}{(k-1)!} u^{\frac{\alpha k-1}{\alpha}}e^{-u}\right)du,\nonumber\\
&=&-\int_{0}^{\infty}\dfrac{1}{(k-1)!} u^{k-1}e^{-u}\log\dfrac{\alpha}{(k-1)!}du,\nonumber\\
 &&-\dfrac{\alpha k-1}{\alpha(k-1)!}\int_{0}^{\infty}u^{k-1}e^{-u}\log u du +\int_{0}^{\infty}\dfrac{u^{k}e^{-u}}{(k-1)!}du,\nonumber\\
&=&-\log\dfrac{\alpha}{(k-1)!}+I_A+\dfrac{1}{(k-1)!}\Gamma(k+1),\label{ent_g2}
\end{eqnarray}
where 
\begin{eqnarray}
I_A&=&-\dfrac{\alpha k-1}{\alpha(k-1)!}\int_{0}^{\infty}u^{k-1}e^{-u}\log u du = -\dfrac{\alpha k-1}{\alpha(k-1)!}A(k),\nonumber \\
&=&-\dfrac{\alpha k-1}{\alpha}\left(-\gamma+\sum_{i=1}^{k-1}\dfrac{1}{i}\right), \label{order_weibull_I1}
\end{eqnarray}
using Lemma \ref{Lemma3}. 
From (\ref{ent_g2}) and (\ref{order_weibull_I1}), 
\begin{eqnarray}
H(\psi_\alpha^{(k)})=-\log\dfrac{\alpha}{(k-1)!}-\dfrac{\alpha k-1}{\alpha}\left(-\gamma+\sum_{i=1}^{k-1}\dfrac{1}{i}\right)+\dfrac{\Gamma(k+1)}{(k-1)!}.
\end{eqnarray}
\end{proof}
\begin{proof}[\textbf{Proof Lemma \ref{TPhi_orderEnt} (iii)}] From (\ref{order_density}), $\;\lambda^{(k)}(x)=\dfrac{1}{(k-1)!} e^{-kx}e^{-e^{-x}}, \;$ and the entropy of $\lambda^{(k)}(x)$ is
\begin{eqnarray}
H(\lambda^{(k)})&=&-\int_{-\infty}^{\infty}\dfrac{1}{(k-1)!} e^{-kx}e^{-e^{-x}}\log\left(\dfrac{1}{(k-1)!} e^{-kx}e^{-e^{-x}}\right)dx. \nonumber
\end{eqnarray}
Putting $e^{-x}=u,$ we have $-e^{-x}dx=du$ and 
\begin{eqnarray}
H(\lambda^{(k)}) &=&-\int_{0}^{\infty}\dfrac{1}{(k-1)!} u^{k-1}e^{-u}\log\left(\dfrac{u^{k}e^{-u}}{(k-1)!}\right)du,\nonumber\\
&=&\int_{0}^{\infty}\dfrac{1}{(k-1)!} u^{k-1}e^{-u}\log(k-1)!du,\nonumber\\
&&-\dfrac{k}{(k-1)!}\int_{0}^{\infty}u^{k-1}e^{-u}\log u du+\int_{0}^{\infty}\dfrac{u^{k}e^{-u}}{(k-1)!}du,\nonumber\\
&=&\log(k-1)!+I_A+\dfrac{1}{(k-1)!}\Gamma(k+1),\label{ent_g3}
\end{eqnarray}
where 
\begin{eqnarray}
I_A&=&-\dfrac{k}{(k-1)!}\int_{0}^{\infty}u^{k-1}e^{-u}\log u du =-\dfrac{k}{(k-1)!}A(k),\nonumber \\
&=&-k\left(-\gamma+\sum_{i=1}^{k-1}\dfrac{1}{i}\right), \label{order_gumbel_I1}
\end{eqnarray}
using Lemma \ref{Lemma3}. 
From (\ref{ent_g3}) and (\ref{order_gumbel_I1}), 
\begin{eqnarray}
H(\lambda^{(k)})=\log(k-1)!-k\left(-\gamma+\sum_{i=1}^{k-1}\dfrac{1}{i}\right)+\dfrac{\Gamma(k+1)}{(k-1)!}.
\end{eqnarray}
\end{proof}

\begin{proof}[\textbf{Proof of Theorem \ref{TPhi_order} (i)}]
 Since $F \in \mathcal{D}(\Phi_\alpha),$ from Lemma \ref{Lemma.5}, $\;
\lim_{n \rightarrow \infty} g_n^{(k)}(x) $ \\ $=\phi_\alpha^{(k)}(x)=\dfrac{\alpha}{(k-1)!} x^{-\alpha k-1}e^{-x^{-\alpha}},\;$ and for large $n$ and $x\in [L',L];$ with $L,L' > 0$
\begin{eqnarray}
&&\abs{g_n^{(k)}(x)\log g_n^{(k)}(x)-\phi_\alpha^{(k)}(x)\log \phi_\alpha^{(k)}(x)}<1,\nonumber\\
&\Leftrightarrow& -1+\phi_\alpha^{(k)}(x)\log \phi_\alpha^{(k)}(x)<g_n^{(k)}(x)\log g_n^{(k)}(x)<1+\phi_\alpha^{(k)}(x)\log \phi_\alpha^{(k)}(x).\nonumber
\end{eqnarray}
Since $\;\int_{L'}^{L}(\phi_\alpha^{(k)}(x)\log \phi_\alpha^{(k)}(x)+1)dx<\infty, \;$ by the DCT, for $L',L>0$
\begin{eqnarray}\label{sorder_frechet_en}
\lim_{n\to\infty}H(g_n^{(k)})&=&-\lim_{L\to \infty}\lim_{L'\to 0}\lim_{n\to\infty}\int_{L'}^{L}g_n^{(k)}(x)\log g_n^{(k)}(x)dx,\nonumber\\
&=&-\int_{0}^{\infty}\lim_{n\to\infty}g_n^{(k)}(x)\log g_n^{(k)}(x)dx,\nonumber\\
&=&-\int_{0}^{\infty}\phi_\alpha^{(k)}(x)\log\phi_\alpha^{k}(x)dx,\nonumber\\
&=&H(\phi_{\alpha}^{(k)}).
\end{eqnarray}
\end{proof}
\begin{proof}[\textbf{Proof Theorem \ref{TPhi_order} (ii)}]
Since $F \in \mathcal{D}(\Psi_\alpha),$ from Lemma \ref{Lemma.5}, $\lim_{n \rightarrow \infty} g_n^{(k)}(x)$ $=\psi_{\alpha}^{(k)}(x)=\dfrac{\alpha}{(k-1)!} (-x)^{\alpha k-1}e^{-(-x)^{\alpha}}\;$ and for large $n$ and $x\in [L',L];$ with $L,L' < 0$
\begin{eqnarray}
&&\abs{g_n^{(k)}(x)\log g_n^{(k)}(x)-\psi_\alpha^{(k)}(x)\log \psi_\alpha^{(k)}(x)}<1,\nonumber\\
&\Leftrightarrow& -1+\psi_\alpha^{(k)}(x)\log \psi_\alpha^{(k)}(x)<g_n^{(k)}(x)\log g_n^{(k)}(x)<1+\psi_\alpha^{(k)}(x)\log \psi_\alpha^{(k)}(x).\nonumber
\end{eqnarray}
Since $\;\int_{L'}^{L}(\psi_\alpha^{(k)}(x)\log \psi_\alpha^{(k)}(x)+1)dx<\infty\;,$  
by the DCT, for $L',L<0$
\begin{eqnarray}\label{sorder_weibull_en}
\lim_{n\to\infty}H(g_n^{(k)})&=&-\lim_{L\to 0}\lim_{L'\to-\infty}\lim_{n\to\infty}\int_{L'}^{L}g_n^{(k)}(x)\log g_n^{(k)}(x)dx,\nonumber\\
&=&-\int_{-\infty}^{0}\lim_{n\to\infty}g_n^{(k)}(x)\log g_n^{(k)}(x)dx,\nonumber\\
&=&-\int_{-\infty}^{0}\psi_\alpha^{(k)}(x)\log\psi_\alpha^{k}(x)dx,\nonumber\\
&=&H(\psi_\alpha^{(k)}).
\end{eqnarray}
\end{proof}
\begin{proof}[\textbf{Proof of Theorem \ref{TPhi_order} (iii)}]
 Since $F \in \mathcal{D}(\Lambda),$ from Lemma \ref{Lemma.5}, $\lim_{n \rightarrow \infty} g_n^{(k)}(x)$ $= \lambda^{(k)}(x)=\dfrac{1}{(k-1)!} e^{-kx}e^{-e^{-x}}.\;$ and for large $n$ and $x\in [-L,L];$ with $L> 0$
\begin{eqnarray}
&&\abs{g_n^{(k)}(x)\log g_n^{(k)}(x)-\lambda^{(k)}(x)\log \lambda^{(k)}(x)}<1,\nonumber\\
&\Leftrightarrow& -1+\lambda^{(k)}(x)\log \lambda^{(k)}(x)<g_n^{(k)}(x)\log g_n^{(k)}(x)<1+\lambda^{(k)}(x)\log \lambda^{(k)}(x).\nonumber
\end{eqnarray}
Since $\;\int_{-L}^{L}(\lambda^{(k)}(x)\log \lambda^{(k)}(x)+1)dx<\infty\;,$ 
by the DCT, for $L>0, L^{\prime} < 0,$
\begin{eqnarray}\label{sorder_gumbel_en}
\lim_{n\to\infty}H(g_n^{(k)}(x))&=&-\lim_{L\to\infty}\lim_{L^{\prime}\to-\infty}\lim_{n\to\infty}\int_{L^{\prime}}^{L}g_n^{(k)}(x)\log g_n^{(k)}(x)dx,\nonumber\\
&=&-\int_{-\infty}^{\infty}\lim_{n\to\infty}g_n^{(k)}(x)\log g_n^{(k)}(x)dx,\nonumber\\
&=&-\int_{-\infty}^{\infty}\lambda^{(k)}\log \lambda^{(k)} dx,\nonumber\\
&=&H(\lambda^{(k)}).
\end{eqnarray}
\end{proof}

\appendix
\section[Appendix A]{}\label{more results}
\begin{thm}\label{Moment}
	Let $F\in\mathcal{D}(G).$
\begin{enumerate}
	\item[(i)] If $G=\Phi_\alpha,$ then $a_n=\left(1/(1-F)\right)^{\leftarrow}(n),$ $b_n=0,$ and if for some integer $0<k<\alpha,$
\[\int_{-\infty}^{0}\mid{x}\mid^kF(dx)<\infty,\]
then $\lim_{n\to\infty}E\left(\dfrac{M_n}{a_n}\right)^k=\int_{0}^{\infty}x^k\Phi_\alpha(dx)=\Gamma\left(1-\frac{k}{\alpha}\right).$
	\item[(ii)] If $G=\Psi_\alpha,$ then $a_n=r(F)-\left(1/(1-F)\right)^{\leftarrow}(n),$ $b_n=0,$ and if for some integer $k>0,$
\[\int_{-\infty}^{r(F)}\mid{x}\mid^kF(dx)<\infty,\]
then $\lim_{n\to\infty}E\left(\dfrac{M_n-r(F)}{a_n}\right)^k=\int_{-\infty}^{0}x^k\Psi_\alpha(dx)=(-1)^k\Gamma\left(1+\frac{k}{\alpha}\right).$
	\item[(iii)] If $G=\Lambda,$ then $b_n=\left(1/(1-F)\right)^{\leftarrow}(n),$ $a_n=f(b_n),$ and if for some integer $k>0,$
\[\int_{-\infty}^{r(F)}\mid{x}\mid^kF(dx)<\infty,\]
then $\lim_{n\to\infty}E\left(\dfrac{M_n-b_n}{a_n}\right)^k=\int_{-\infty}^{\infty}x^k\Lambda(dx)=(-1)^k\Gamma^{(k)}(1),$
where $\Gamma^{(k)}(1)$ is the k-th derivative of the Gamma function at $x=1.$
\end{enumerate}
\end{thm}

\begin{thm}\label{thm_von}
\item[(i)] Suppose that df $F$ is absolutely continuous with density $f$ which is eventually positive. \\
\item[(a)] If for some $\alpha>0$
\begin{eqnarray}\label{von_F}
\lim_{x\to\infty}\dfrac{xf(x)}{\overline{F}(x)}=\alpha
\end{eqnarray}
then $F\in\D(\Phi_\alpha).$\\
\item[(b)] If $f$ is nonincreasing and $F\in\D(\Phi_\alpha)$ then (\ref{von_F}) holds.
\item[(ii)] Suppose $F$ has right endpoint $r(F)$ finite and density $f$ positive in a left neighbourhood of $r(F).$  \\
\item[(a)] If for some $\alpha>0$
\begin{eqnarray}\label{von_W}
\lim_{x\to r(F)}\dfrac{(r(F)-x)f(x)}{\overline{F}(x)}=\alpha
\end{eqnarray}
then $F\in\D(\Psi_\alpha).$\\
\item[(b)] If $f$ is nonincreasing and $F\in\D(\Psi_\alpha)$ then (\ref{von_W}) holds.
\end{thm}

\begin{lem}\label{lem1_appendix}
Let $u(x)$ be absolutely continuous auxiliary function with $u(x)\to 0$ as $x\uparrow r(F).$
	
	\item[(a)] If $r(F)=\infty$ then $\lim_{t\to\infty}t^{-1}u(t)=0.$
	
	\item[(b)] If $r(F)<\infty$ then $u(r(F))=\lim_{t\uparrow r(F)}u(t)=0$ and $\lim_{n\to r(F)}(r(F)-t)^{-1}u(t)=0.$
	
	In either case $\lim_{t\uparrow r(F)}(t+xu(t))=r(F)$ for all $x\in\Real.$
\end{lem}
\begin{thm}\label{lem2_appendix}
	If $u$ satisfies the conditions given in Lemma (\ref{lem1_appendix}), then 
\[\lim_{t\to r(F)}\dfrac{u(t+xu(t))}{u(t)}=1 \;\; \text{locally uniformly in}\; x\in\Real.\] 
\end{thm}
\begin{thm} Let  $F$ be absolutely continuous in a left neighbourhood of $r(F)$ with density $f.$ If 
\begin{eqnarray} \label{lem3_appendix} \lim_{x\uparrow r(F)}f(x)\int_{x}^{r(F)}\overline{F}(t)dt/\overline{F}(x)^2=1,\end{eqnarray} 
	 then $F\in\D(\Lambda).$ In this case we may take,
\[u(t)=\int_{x}^{r(F)}\overline{F}(s)ds/\overline{F}(t), \;\;b_n=F^{-}(1-1/n),\;\;a_n=u(b_n).\]
\end{thm}
\begin{thm}\label{lem4_appendix}
Suppose that $U_1,\;U_2$ are nondecreasing and $\rho$- varying, $0<\rho<\infty.$ Then for $0\leq c\leq \infty,$   $U_1(x)\sim cU_2(x)$ iff $U_1^{\leftarrow}(x)\sim c^{-\rho^{-1}}U_2^{\leftarrow},$ as $n\to\infty.$
\end{thm}
\begin{defn}\label{def1_appendix}
A nondecreasing function $U=1/\overline{F}$ is $\Gamma$ varying if $U$ is defined on $(l(F),r(F)),$ with  $\lim_{x\to r(F)}U(x)=\infty$ and there exists a positive function $f$ defined on $(l(F),r(F))$ such that for all $x$
\begin{eqnarray}\label{e1_appendix}
\lim_{t\to r(F)}\dfrac{U(t+xf(t))}{U(t)}=e^x.
\end{eqnarray}
The function $f$ is called an auxiliary function and is unique up to asymptotic equivalence. If (\ref{e1_appendix}) is satisfied by both $f_1$ and $f_2$ then $\;F_t(f_i(t)x)\to 1-e^{-x}\;$ as $\;t \to \infty,\;i = 1,2,$ where $\;F_t(x)=1-U(t)/U(t+x)\;$ is a family of distributions and we have $f_1(t)\sim f_2(t).$
\end{defn}
\begin{thm}\label{gn_conv}
Suppose that $F$ is absolutely continuous with pdf $f.$ If $F\in\D(G)$ and
\begin{itemize}
\item[(i)] $G=\Phi_\alpha,$ then $g_n(x)\to \phi_\alpha(x)$ locally uniformly on $(0,\infty)$ iff (\ref{von_F}) holds;
\item[(ii)] $G=\Psi_\alpha,$ then $g_n(x)\to \psi_\alpha(x)$ locally uniformly on $(-\infty,0)$ iff (\ref{von_W}) holds;
\item[(iii)] $G=\Lambda,$ then $g_n(x)\to \lambda(x)$ locally uniformly on $\Real$ iff (\ref{lem3_appendix}) holds.
\end{itemize}
\end{thm}

\begin{figure}[!ht]
\section{Graphs of densities and relative entropies}  \label{Graphs}
  \centering
   \subfloat[]{\label{pareto_graph_f}\includegraphics[width=0.5\textwidth]{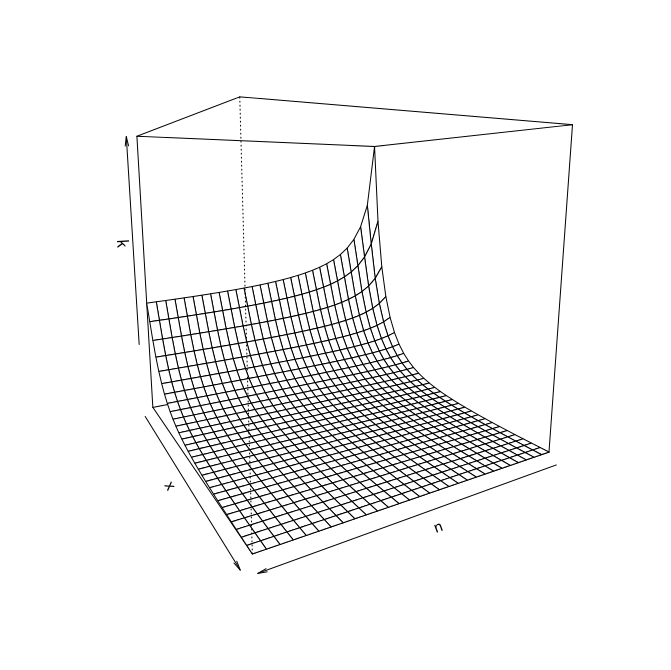}} 
  \subfloat[]{\label{pareto_graph_ent}\includegraphics[width=0.5\textwidth]{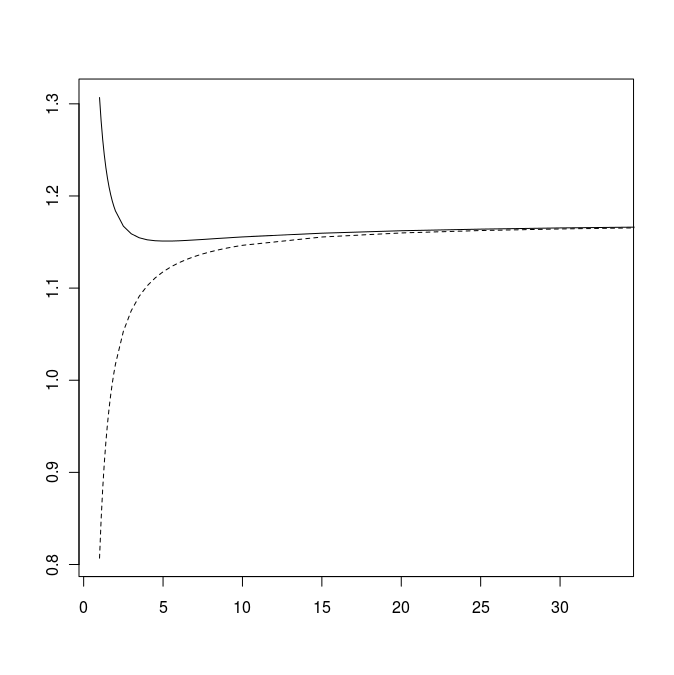}}                
  \caption{\small{(A) Graph of $g_n(x)$ in the Pareto case with $\alpha=2,$ $1 < n < 5$ and $x\in (1,5)$). (B) Entropy $H(g_n)$ (dashed line) and  $\Delta_g(g_n)$ (line) for  $1 < n < 100.$}}

\centering
  \subfloat[]{\includegraphics[width=0.5\textwidth]{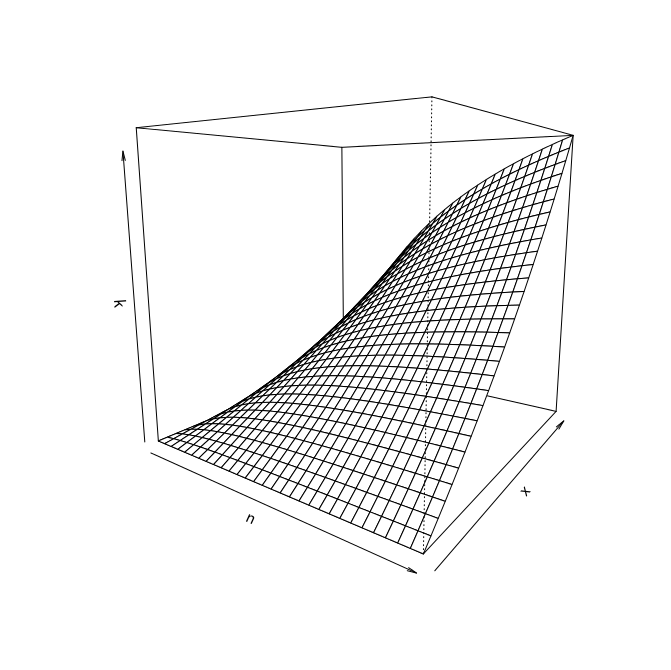}}
  \subfloat[]{\includegraphics[width=0.5\textwidth]{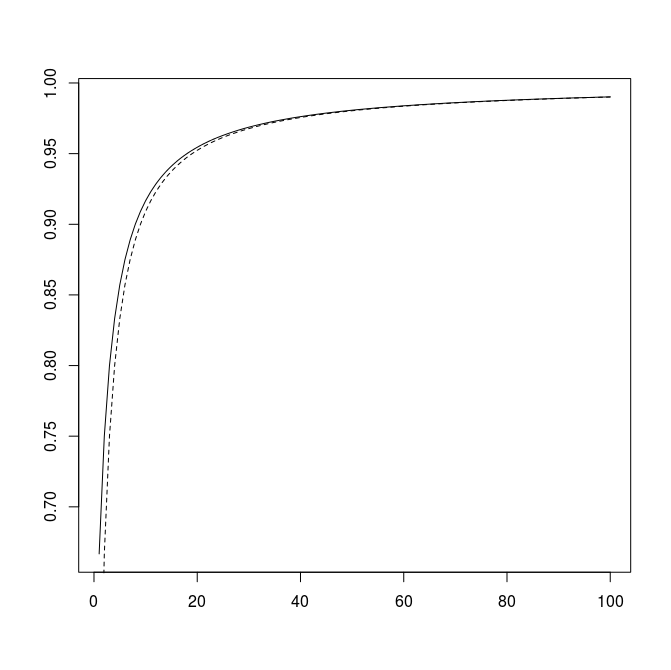}}
  \caption{\small{(A) Graph of $g_n(x)$ in the uniform case with $1 < n < 5$ and $x\in (0,1).$ (B) Entropy $H(g_n)$ (dashed line) and $\Delta_g(g_n)$ (line) for  $1 < n < 100.$}}
\end{figure}

\begin{figure}[!ht]
  \centering
   \subfloat[]{\label{exp_graph_f}\includegraphics[width=0.5\textwidth]{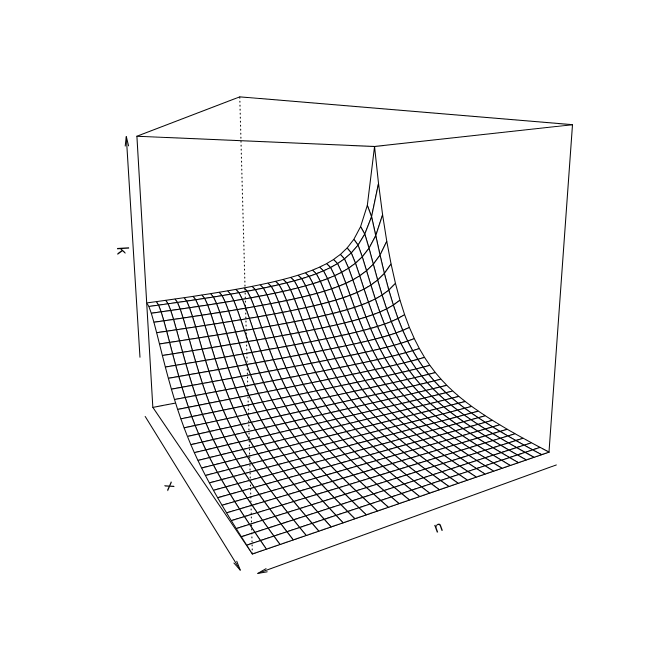}} 
  \subfloat[]{\label{exp_graph_ent}\includegraphics[width=0.5\textwidth]{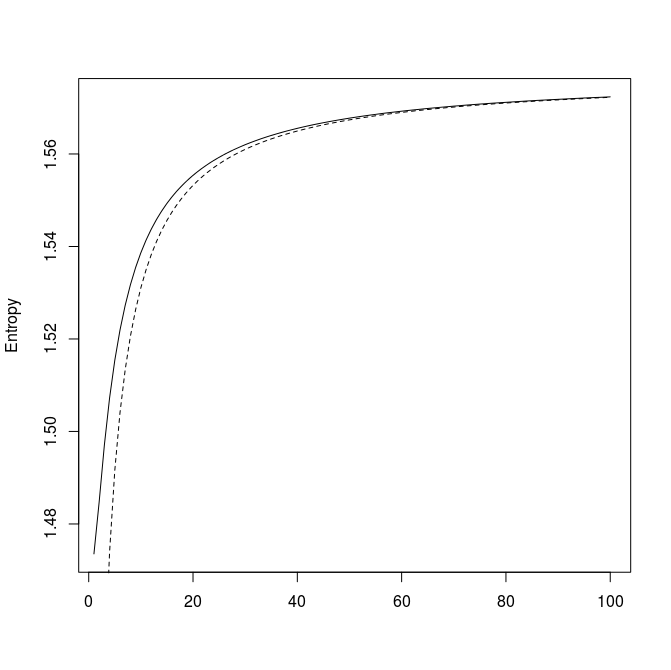}}                
  \caption{\small{(A) Graph of $g_n(x)$ in the exponential case with $1 < n < 5$ and $x\in (0,5).$ (B) Entropy $H(g_n)$ (dashed line) and $\Delta_g(g_n)$ (line) for  $1 < n < 100.$}}

  \centering
  \subfloat[]{\label{normal_graph_f}\includegraphics[width=0.5\textwidth]{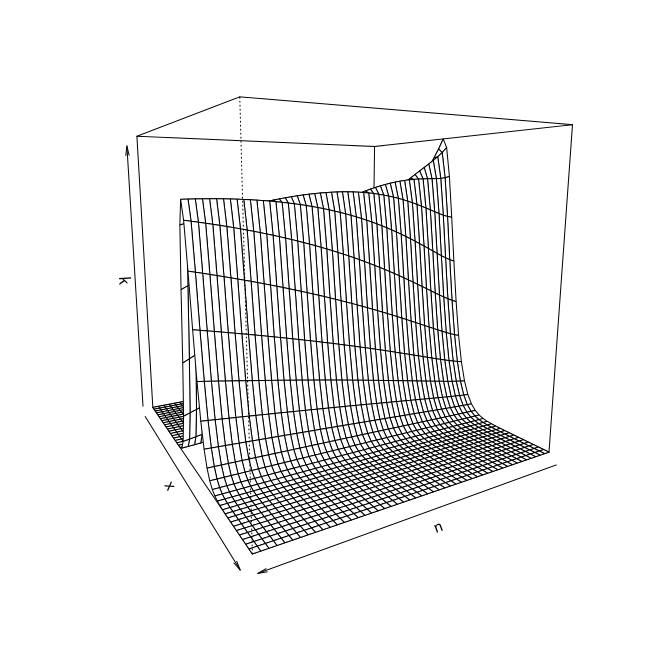}}
  \subfloat[]{\label{normal_graph_ent}\includegraphics[width=0.5\textwidth]{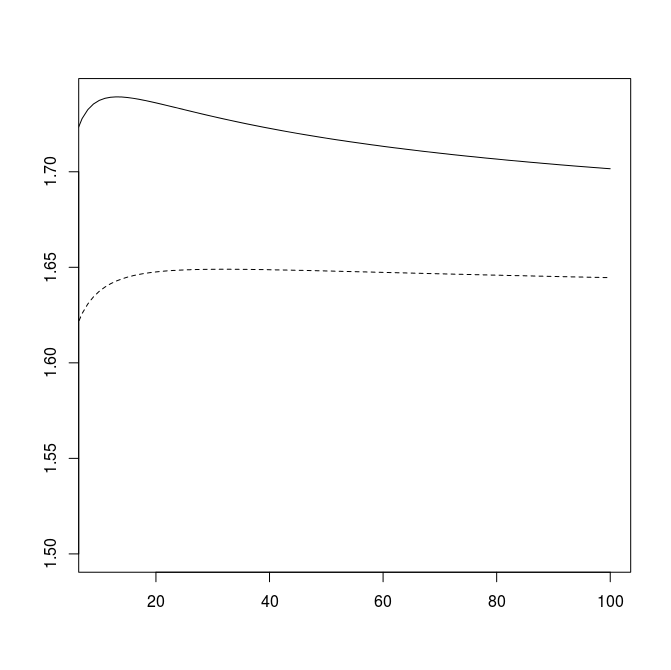}}
  \caption{\small{(A) Graph of $g_n(x)$ in the normal case ($2 < n < 10$ and $x\in (-10,10)$). (B) Entropy $H(g_n)$ (dashed line) and  $\Delta_g(g_n)$ (line) for $2 < n < 100.$} In (A), $g_n$ is increasing for fixed $x$ and varying $n.$}
\end{figure}

\end{document}